\documentclass{article}
\usepackage{fullpage}

\usepackage[utf8]{inputenc}
\usepackage{amsthm,amssymb}
\usepackage{amsmath}
\usepackage{mathtools}
\usepackage{caption,subcaption}
\usepackage{epsfig,graphicx,graphics,color,tikz,tikz-cd}
\usetikzlibrary{fit}
\usepackage [linesnumbered, ruled,vlined]{algorithm2e}
\usetikzlibrary{calc,decorations.pathmorphing,shapes}
\usepackage{tcolorbox}
\usepackage{array,makecell,multirow}
\setcellgapes{4pt}
\usepackage{bbm}
\usepackage{varwidth}
\usepackage{rotating}
\usepackage{booktabs}
\usepackage[mathcal,mathscr]{eucal}
\usepackage{varwidth}
\usepackage{enumerate}
\usepackage[shortlabels]{enumitem}
\usepackage[sort]{cite}
\usepackage{xspace}
\usepackage{bm}
\usepackage{csquotes}
\usepackage{titling}
\usepackage{aliascnt}
\usepackage{hyperref}
\hypersetup{
colorlinks=true,       
linkcolor=blue,        
citecolor=magenta,         
filecolor=magenta,     
urlcolor=cyan,         
linktoc=all
}

\theoremstyle{plain}
\newtheorem{thm}{Theorem}[section]
\newtheorem*{thm*}{Theorem}
\newaliascnt{lem}{thm}
\newtheorem{lem}[lem]{Lemma}
\aliascntresetthe{lem}
\newaliascnt{cor}{thm}
\newtheorem{cor}[cor]{Corollary}
\aliascntresetthe{cor}
\newaliascnt{cl}{thm}
\newtheorem{cl}[cl]{Claim}
\aliascntresetthe{cl}
\newaliascnt{prop}{thm}
\newtheorem{prop}[prop]{Proposition}
\aliascntresetthe{prop}

\theoremstyle{definition}
\newaliascnt{rem}{thm}
\newtheorem{rem}[rem]{Remark}
\aliascntresetthe{rem}

\usepackage[noabbrev]{cleveref}

\Crefname{thm}{Theorem}{Theorems}
\Crefname{lem}{Lemma}{Lemmas}
\Crefname{cor}{Corollary}{Corollaries}
\Crefname{cl}{Claim}{Claims}
\Crefname{prop}{Proposition}{Propositions}
\Crefname{rem}{Remark}{Remarks}

\newcommand*{\claimproofname}{Proof of claim.}
\newenvironment{claimproof}[1][\claimproofname]{\begin{proof}[#1]}{\end{proof}}

\def\final{0}  
\def\iflong{\iffalse}
\ifnum\final=0  
\newcommand{\knote}[1]{{\color{red}[{\tiny \textbf{Kristóf:} \bf #1}]\marginpar{\color{red}*}}}
\newcommand{\cnote}[1]{{\color{blue}[{\tiny \textbf{Chaitanya:} \bf #1}]\marginpar{\color{blue}*}}}
\newcommand{\ynote}[1]{{\color{purple}[{\tiny \textbf{Yuhang:} \bf #1}]\marginpar{\color{purple}*}}}
\else 
\newcommand{\knote}[1]{}
\newcommand{\cnote}[1]{}
\newcommand{\ynote}[1]{}
\fi

\makeatletter
\renewcommand{\paragraph}{%
  \@startsection{paragraph}{4}%
  {\z@}{1.5ex \@plus 1ex \@minus .2ex}{-1em}%
  {\normalfont\normalsize\bfseries}%
}
\makeatother

\DeclareMathOperator\si{si}
\DeclareMathOperator{\clo}{cl}
\DeclareMathOperator{\spa}{span}

\newcommand{\cB}{\mathcal{B}}
\newcommand{\cC}{\mathcal{C}}

\newcommand{\cH}{\mathcal{H}}
\newcommand{\cI}{\mathcal{I}}

\newcommand{\cM}{M}
\newcommand{\cN}{N}
\newcommand{\cP}{\mathcal{P}}

\definecolor[named]{Blue}{cmyk}{1,0.1,0,0.1}
\definecolor[named]{Yellow}{cmyk}{0,0.16,1,0}
\definecolor[named]{Orange}{cmyk}{0,0.42,1,0.01}
\definecolor[named]{Red}{cmyk}{0,0.90,0.86,0}
\definecolor[named]{LightBlue}{cmyk}{0.49,0.01,0,0}
\definecolor[named]{Green}{cmyk}{0.20,0,1,0.19}

\SetKwFor{Whileblock}{while}{do}{}
\SetKwBlock{Begin}{}{}
\SetVlineSkip{1.5pt}
\makeatletter
\newcommand{\linkdest}[1]{\Hy@raisedlink{\hypertarget{#1}{}}}
\makeatother

\newlength{\bibitemsep}
\newlength{\bibparskip}
\let\oldthebibliography\thebibliography
\renewcommand\thebibliography[1]{%
\oldthebibliography{#1}%
\setlength{\parskip}{\bibitemsep}%
\setlength{\itemsep}{\bibparskip}%
}

\title{Joins and Ear Decompositions Beyond Graphic Matroids}

\thanksmarkseries{alph}
\author{
Yuhang Bai\thanks{School of Mathematics and Statistics, Northwestern Polytechnical University and Xi'an-Budapest Joint Research Center for Combinatorics, Xi'an 710129,
Shaanxi, People's Republic of China. Email: \texttt{yhbai@mail.nwpu.edu.cn}.}
\and
Kristóf Bérczi\thanks{MTA-ELTE Matroid Optimization Research Group and HUN-REN–ELTE Egerváry Research Group, Department of Operations Research, ELTE Eötvös Loránd University, and HUN-REN Alfréd Rényi Institute of Mathematics, Budapest, Hungary. Email: \texttt{kristof.berczi@ttk.elte.hu}.}
\and
Chaitanya Nalam\thanks{University of Michigan, Ann Arbor, USA. Email: \texttt{nalamsai@umich.edu}.}
}
\date{}

\begin{document}
\maketitle
\thispagestyle{empty}

\begin{abstract}

For a matroid $M$, a join is a set $J\subseteq E(M)$ that meets every circuit $C$ in at most $|C|/2$ elements. Let $\mu(M)$ denote the maximum size of a join. Motivated by Frank's min--max theorem for graphic matroids, we compare $\mu(M)$ with an ear-decomposition parameter $\eta(M)=(r(M)+\varphi(M))/2$, where $\varphi(M)$ is the minimum number of even lobes in an ear decomposition of $M$. Frank's theorem implies $\mu(M)=\eta(M)$ for connected graphic matroids. 

Here we study how far this equality extends beyond graphic matroids. We show that the exact equality does not hold in general: it already fails for cographic matroids, hence within the binary class. Furthermore, the class of matroids satisfying $\mu(M)=\eta(M)$ is not minor-closed, thus there is little hope for a forbidden minor characterization. We also prove that computing a maximum join is NP-hard for cographic matroids, hard to approximate within a factor of $519/520$, and NP-hard for sparse paving matroids given by their list of bases. Despite these negative results, we show that the two parameters remain quantitatively comparable in several natural classes. We prove comparison bounds for binary, paving, cographic, and arbitrary connected matroids. In particular, using Seymour's decomposition theorem, we combine the equality for graphic matroids, the bound for cographic matroids, and a direct analysis of $R_{10}$ to obtain $\eta(M)\leq 6\mu(M)-2$ for every regular matroid $M$.

\medskip

\end{abstract}

\newpage
\pagenumbering{roman}
\tableofcontents
\newpage
\pagenumbering{arabic}
\setcounter{page}{1}

\section{Introduction}
\label{sec:intro}

Ear decompositions are one of the basic structural tools in graph theory and matching theory. They provide inductive characterizations of several fundamental graph classes: connected graphs without bridges are precisely the graphs admitting an ear-decomposition~\cite{Robbins1939}, 2-vertex-connected graphs are characterized by open ear-decompositions~\cite{Whitney1932}, and factor-critical graphs are characterized by odd ear-decompositions~\cite{Lovasz1972}. The latter theorem, due to Lovász, states that a graph $G$ is factor-critical if and only if $G$ admits an ear-decomposition in which every ear has odd length. Thus the minimum possible number of even ears in an ear-decomposition gives a natural measure of how far a graph is from being factor-critical.

Frank~\cite{frankConservative} made this measure part of a min--max theorem on conservative weightings and ear decompositions. Let $G=(V,E)$ be a $2$-edge-connected graph. A set $J\subseteq E$ is called a \emph{join} if
\begin{equation*}
        |J\cap C|\leq \frac{|C|}{2}
\end{equation*}
for every circuit\footnote{We use \emph{circuit} rather than \emph{cycle} for graphs, since in matroid theory a cycle means a disjoint union of circuits.} $C$ of $G$. Let $\mu(G)$ denote the maximum cardinality of a join.

A weighting $w\colon E\to \mathbb R$ is called \emph{conservative} if $\sum_{e\in C} w(e)\ge 0$ for every circuit $C$ of $G$. To a set $J\subseteq E$, associate the $\pm1$ weighting $w_J$ that assigns weight $-1$ to the edges in $J$ and weight $+1$ to all remaining edges. Then $J$ is a join if and only if $w_J$ is conservative. Thus joins are equivalently conservative $\pm1$ edge-weightings, with the edges of the join being precisely the negative edges. In particular, $\mu(G)$ is the maximum number of negative edges in a conservative $\pm1$ edge-weighting. Frank also gave an equivalent interpretation of $\mu(G)$ in terms of minimum $T$-joins~\cite{frankConservative}. A further interpretation, in terms of the covering radius of the cycle code of $G$, was proved by Sol\'e and Zaslavsky~\cite{SoleZaslavsky1994}.

On the ear-decomposition side, let $\varphi(G)$ be the minimum number of even ears in an ear-decomposition of $G$. Since every ear-decomposition of a connected graph has $|E|-|V|+1$ ears, minimizing the number of even ears is equivalent to maximizing the number of odd ears. If $\gamma(G)$ denotes the maximum number of odd ears in an ear-decomposition of $G$, then $\gamma(G)=|E|-|V|+1-\varphi(G)$. Frank proved the following min--max theorem:
\begin{equation*}
        \mu(G)=\frac{|V|-1+\varphi(G)}{2}
              =\frac{|E|-\gamma(G)}{2}.
\end{equation*}
In other words, the maximum size of a join is exactly determined by the optimum parity profile of an ear-decomposition. It is worth emphasizing that the proof in~\cite{frankConservative} is constructive and yields polynomial-time algorithms for constructing the corresponding optima.

The above min--max theorem has also played a role in the algorithmic use of ear decompositions. Its formulation connects ear-decompositions with conservative weightings, minimum $T$-joins, and matchings. This viewpoint later became part of the toolkit for approximation algorithms for graphic traveling-salesman-type problems. In particular, the work of Sebő and Vygen~ \cite{SeboVygen2014} on ``nicer ears'' uses optimized ear-decompositions as a central ingredient in improved approximation algorithms for graph-TSP, the minimum $T$-tour problem, the $s$-$t$ path graph-TSP, and the minimum-size two-edge-connected spanning subgraph problem. Though these results are not used in the paper, they illustrate that optimized ear-decompositions carry algorithmically useful information.

Szegedy and Szegedy~\cite{Szegedy2006} later developed a matroidal analogue of the factor-critical and ear-decomposition side of this theory. Let $\cM$ be a matroid on ground set $E$. An ear-decomposition of $\cM$ is a sequence of circuits $C_1,\dots,C_t$ such that, writing
\begin{equation*}
        D_i=C_1\cup \dots \cup C_i
        \qquad\text{and}\qquad
        D_0=\emptyset,
\end{equation*}
the set
\begin{equation*}
        P_i=C_i\setminus D_{i-1}
\end{equation*}
of new elements is nonempty, the circuit $C_i$ meets $D_{i-1}$ for $i\geq 2$, and $P_i$ is minimal among the possible new parts of circuits satisfying these two nontriviality conditions. We call $P_i$ the $i^{th}$ \emph{lobe}. The lobe is odd or even according to the parity of $|P_i|$.

As in graphs, define $\varphi(\cM)$ to be the minimum number of even lobes in an ear-decomposition of $\cM$, and define $\gamma(\cM)$ to be the maximum number of odd lobes. For connected matroids, every ear-decomposition has the same number of lobes, namely $|E(\cM)|-r(\cM)$, and the lobe sizes satisfy $r(\cM)=\sum_i (|P_i|-1)$.
It follows that
\begin{equation*}
        \frac{|E(\cM)|-\gamma(\cM)}{2}
        =
        \frac{r(\cM)+\varphi(\cM)}{2}.
\end{equation*}
We denote this common value by $\eta(\cM)$.

Building on Szigeti's~\cite{Szigeti1996} work on the matroid defined by optimal ear-decompositions of a graph, Szegedy and Szegedy defined a bridgeless matroid to be factor-critical if it admits an odd ear-de\-com\-po\-si\-tion, or equivalently, if $\varphi(\cM)=0$. Their main theorem gives an algebraic characterization of this property for matroids representable over fields of characteristic $2$: such a matroid admits an odd ear-decomposition if and only if it has a representation by a space on which the induced scalar product is a non-degenerate symplectic form. They also show that, for matroid $\cM$ representable over characteristic $2$, the independent sets of $\cM$ whose contraction admits an odd ear-decomposition form the feasible sets of a representable $\Delta$-matroid. In particular, using the approach developed in Szegedy's thesis~\cite{SzegedyThesis2005}, the value $\varphi(\cM)$ can be computed in randomized polynomial time for matroids represented over finite fields of characteristic $2$.

More recently, ear-decomposition ideas have appeared in two further matroidal settings. Jordán~\cite{jordan2024ear} used ear decompositions of connected matroids to study minimally connected matroids and rigidity matroids, obtaining sharp bounds for minimally connected two-dimensional rigidity matroids and a new proof of Murty's theorem~\cite{murty1974extremal}. In a different, topological direction, Athanasiadis and Ferroni~\cite{athanasiadis2026convex} proved that the augmented Bergman complex of every matroid admits a convex ear decomposition. Although this is a convex ear decomposition of a simplicial complex rather than an ear decomposition of a matroid in the sense used here, it further illustrates the recent role of ear-decomposition methods in matroid theory.

These results give a matroidal extension of the ear-decomposition side of Frank's theorem, but they do not address the join side. For a matroid $\cM$, the natural definition is identical to the graphic one: a set $J\subseteq E(\cM)$ is a \emph{join} if
\begin{equation*}
        |J\cap C|\leq \frac{|C|}{2}
\end{equation*}
for every circuit $C$ of $\cM$, and $\mu(\cM)$ denotes the maximum size of a join. If $\cM$ is graphic, then Frank's theorem says precisely that $\mu(\cM)=\eta(\cM)$.
This parameter has a separate interpretation through coding theory and signed graphs. For a binary matroid $\cM$, $\mu(\cM)$ is exactly  the covering radius of the cycle space. For a cographic matroid of a graph $G$, this becomes the covering radius of the cutset code of $G$. Equivalently, it is the maximum frustration of $G$: the largest, over all signings of $G$, of the minimum number of negative edges obtainable after switching~\cite{SoleZaslavsky1994,Bowlin2012}. This connection is central to our hardness proof. It also explains why the cographic case is already algorithmically different from the graphic one.

Thus Frank's theorem relates two notions that both have natural matroidal analogues: joins, defined by inequalities on circuits, and ear decompositions, which are available for connected matroids. This makes it natural to ask whether the min--max equality is special to graphs or can be extended to a more general matroidal setting. Our results show that the exact min--max theorem is genuinely graphic, but that its two sides remain meaningfully related in several large matroid classes; in particular, regular matroids still admit a constant-factor analogue. The central question of this paper is therefore the following: 
\begin{center}
    \emph{How much of the join–ear equality survives beyond graphic matroids?}
\end{center} 

\subsection{Our Results}
\label{sec:results}

Our answer to the question above has three parts. First, the exact join--ear equality is special to graphic matroids and does not extend to matroids in general. Second, the two sides of Frank's theorem behave differently from an algorithmic point of view: the ear-decomposition parameter is tractable in an important representable setting, while the join parameter is already hard for closely related matroid classes. Third, despite these negative results, the two parameters remain quantitatively comparable in several natural classes of matroids.

We begin by showing that the equality $\mu=\eta$ is not a general matroidal phenomenon. It already fails for cographic matroids: the matroid $\cM=\cM^\ast(K_{4,4})$ satisfies $\mu(\cM)=4$ and $\eta(\cM)=5$ (\Cref{prop:k44-example}). A one-edge extension of this example shows that the class of connected matroids satisfying $\mu=\eta$ is not minor-closed (\Cref{rem:minor}). The reverse inequality can also occur: we give a connected paving matroid with $\eta(\cM)<\mu(\cM)$ (\Cref{prop:incomparable}). Thus, for general matroids, the two parameters are not ordered in either direction.

We then turn to the algorithmic side. The ear-decomposition parameter admits a positive algorithmic result in characteristic $2$: the results of Szegedy and Szegedy~\cite{Szegedy2006}, together with the algorithmic discussion in Szegedy's thesis~\cite{SzegedyThesis2005}, imply that $\varphi(\cM)$ can be computed in randomized polynomial time for matroids represented over finite fields of characteristic $2$. The join parameter behaves differently. For binary matroids, $\mu(\cM)$ is the covering radius of the cycle space; for cographic matroids, this is the covering radius of a cutset code and, equivalently, the maximum frustration of a graph, following the work of Sol\'e and Zaslavsky~\cite{SoleZaslavsky1994} and Bowlin~\cite{Bowlin2012}. Since Guruswami, Micciancio, and Regev~\cite{GuruswamiMicciancioRegev2005} proved that the covering-radius problem for linear codes is hard to approximate, one cannot expect the same algorithmic behavior as for the ear-decomposition parameter. In our setting, we prove that computing a maximum join is NP-hard already for cographic matroids (\Cref{thm:frustration-hard}). More strongly, unless $\mathrm{P}=\mathrm{NP}$, there is no polynomial-time algorithm that, given a cographic matroid $\cM$, outputs a join of size at least $519/520\cdot\mu(\cM)$ (\Cref{thm:no-approx}). We also prove a separate NP-hardness result for connected sparse paving matroids given by their list of bases (\Cref{thm:paving-mu-hard}).

Since the exact equality fails and the join parameter is hard to compute, it is natural to ask whether $\mu(\cM)$ and $\eta(\cM)$ are at least within a multiplicative factor of each other. We prove several results of this form. For binary matroids, the comparison is one-sided. We prove
\begin{equation*}
        \mu(\cM)\leq \eta(\cM)
\end{equation*}
(\Cref{thm:binary-bound}). Thus, although equality may fail, the join parameter is still bounded above by the ear-decomposition parameter throughout the binary class.

For paving matroids, the two parameters are comparable by absolute constants. We prove
\begin{equation*}
        \frac{2}{3}\mu(\cM)\leq \eta(\cM)\leq 2\mu(\cM),
\end{equation*}
and both constants are best possible (\Cref{thm:paving-bound}).

For arbitrary matroids, no absolute constant comparison holds at this level of generality, but there are tight rank-dependent bounds
\begin{equation*}
        \left\lceil \frac{\mu(\cM)}{2}\right\rceil
        \leq
        \eta(\cM)
        \leq
        \left\lceil \frac{r(\cM)}{2}\right\rceil \mu(\cM)
\end{equation*}
(\Cref{thm:general-bound}), where both bounds are attained.

The main positive result of the paper is a constant-factor analogue of Frank's theorem for regular matroids. We first prove the required estimate for cographic matroids: for every connected cographic matroid,
\begin{equation*}
        \eta(\cM)\leq c_0\mu(\cM),
\end{equation*}
where $c_0<5.5$ (\Cref{thm:cographic-eta-mu-bound}). For regular matroids, we combine Frank's graphic theorem~\cite{frankConservative}, the cographic bound, and the analysis of $R_{10}$ (\Cref{prop:r10-example}) with Seymour's decomposition theorem~\cite{seymour1980decomposition} to obtain
\begin{equation*}
        \eta(\cM)\leq 6\mu(\cM)-2
\end{equation*}
for regular matroids (\Cref{cor:regular}).

Taken together, these results give a precise picture of what remains of Frank's theorem beyond graphic matroids. The exact min--max equality does not survive as a general matroidal statement, and the join parameter becomes computationally difficult even for cographic matroids. Nevertheless, for several fundamental classes of matroids, including binary, paving, cographic, and regular matroids, the join and ear-decomposition parameters remain quantitatively comparable.

\subsection{Organization}
\label{sec:organization}

The paper is organized so as to first explain why one cannot expect a direct matroidal extension of Frank's min--max theorem, and then to develop the approximation results that remain possible. Thus, although the main positive result concerns regular matroids, we begin with the limitations of the equality: counterexamples, non-minor-closedness, and hardness. The hardness proofs are somewhat technical and are logically independent of most of the later comparison bounds; readers mainly interested in the structural approximation results may read the statements of the hardness results in \Cref{sec:limits}, skip their proofs, and then proceed to \Cref{sec:other,sec:regular}.

\Cref{sec:prelim} gives an overview of the required background on matroids, paving matroids, Seymour's decomposition theorem, ear decompositions, and the connection between cographic joins, covering radius, and maximum frustration. In \Cref{sec:limits}, we show that Frank's equality does not extend to general matroids: we give cographic and paving examples separating $\mu$ and $\eta$, discuss non-minor-closedness of the equality, and prove hardness results for maximum joins in cographic and sparse paving matroids. \Cref{sec:other} proves comparison bounds for binary, paving, and arbitrary connected matroids. \Cref{sec:regular} is devoted to regular matroids: after preliminary reductions, we establish the cographic bound, consider the matroid $R_{10}$, and combine these ingredients with Seymour's decomposition theorem to obtain the regular-matroid bound. We conclude the paper in \Cref{sec:conclusion}.

\section{Preliminaries}
\label{sec:prelim}

\paragraph{Basic notation and definitions.}

Given a ground set $E$, the \emph{difference} of $X,Y\subset E$ is denoted by $X-Y$. If $Y$ consists of a single element $y$, then $X-\{y\}$ and $X\cup\{y\}$ are abbreviated as $X-y$ and $X+y$, respectively. The \emph{symmetric difference} of $X$ and $Y$ is defined as $X\triangle Y\coloneqq (X\setminus Y)\cup(Y\setminus X)$.

\paragraph{Graphs.}

For a graph $G=(V,E)$, let $\delta_G(v)$ denote the \emph{set of edges incident to a vertex} $v\in V$, and let $d_G(v)=|\delta_G(v)|$ be the \emph{degree} of $v$. More generally, for $X\subseteq V$, let $\delta_G(X)$ denote the set of edges with exactly one endpoint in $X$. We omit the subscript when the graph is clear from the context. A \emph{cut} of $G$ is a set of the form $\delta_G(X)$ for some $X\subseteq V$. An inclusionwise minimal nonempty cut is called a \emph{bond}. The complete graph on $n$ vertices is denoted by $K_n$, and the complete bipartite graph with parts of size $m$ and $n$ is denoted by $K_{m,n}$. A \emph{cubic graph} is a graph in which every vertex has degree three. For a graph or hypergraph $G$, the \emph{independence number} $\alpha(G)$ is the maximum size of a stable set of vertices in $G$, and the \emph{vertex cover number} $\tau(G)$ is the minimum size of a set of vertices meeting every edge or hyperedge of $G$. We will use the following well-known identity. 

\begin{thm}[Gallai~\cite{Gallai1959,West2001}] \label{thm:gallai} 
Let $G$ be a finite graph, or more generally a finite hypergraph with no empty hyperedge. Then $\alpha(G)+\tau(G)=|V(G)|$.
\end{thm}

\paragraph{Matroids.}

For basic definitions on matroids, we refer the reader to~\cite{oxleyMatroid}. A \emph{matroid} $\cM=(E,\cI)$ is defined by its \emph{ground set} $E$ and its \emph{family of independent sets} $\cI\subseteq 2^E$ that satisfies the \emph{independence axioms}: (I1) $\emptyset\in\cI$, (I2) $X\subseteq Y,\ Y\in\cI\Rightarrow X\in\cI$, and (I3) $X,Y\in\cI,\ |X|<|Y|\Rightarrow\exists e\in Y\setminus X\ s.t.\ X+e\in\cI$. Members of $\cI$ are called \emph{independent}, while sets not in $\cI$ are called \emph{dependent}. The \emph{rank} $r_\cM(X)$ of a set $X$ is the maximum size of an independent set in $X$. The maximal independent subsets of $E$
are called \emph{bases}. 

A \emph{circuit} is an inclusionwise minimal dependent set, while a \emph{loop} is a circuit consisting of a single element. For any basis $B$ and nonloop element $e\in E\setminus B$, $B+e$ contains a unique circuit that is called the \emph{fundamental circuit of $e$ with respect to $B$}.  
Two elements $e, f \in E$ are \emph{parallel} if they form a circuit of size two. A circuit
of size three is called a \emph{triangle}. A \emph{cycle} of a matroid is a (possibly empty) subset of its ground set which can be partitioned into circuits. For a matroid $\cM$, we denote its families of independent sets, bases and circuits by $\cI(\cM)$, $\cB(\cM)$ and $\cC(\cM)$, respectively. 
The \emph{dual} of $\cM$ is the matroid $\cM^*$ with $\cB(\cM^*)=\{B\subseteq E\mid E\setminus B\in\cB(\cM)\}$. 

\begin{thm}[Bryant, Dawson and Perfect~\cite{bryant1978hereditary}]
\label{thm:triangle-circuit}
Let $\cM$ be a matroid in which every independent $2$-set is contained in a $3$-circuit. Then every independent set $A\subseteq E(\cM)$ with $|A|\ge 2$ is contained in a circuit of size $|A|+1$.
\end{thm}

Define $\clo_M(X)=\{x\in E\colon r_M(X\cup x)=r_M(X)\}$, we call $\clo_M(X)$ the \emph{closure} of $X$ in $\cM$. A set $X$ is called a \emph{flat} if $\clo_M(X)=X$, and is a \emph{hyperplane} if it further satisfies $r_M(X)=r_M(E)-1$. A \emph{circuit-hyperplane} is a hyperplane that is also a circuit. We denote by $\si(\cM)$ the \emph{simplification} of $\cM$, which is obtained from $\cM$ by deleting all loops and replacing each parallel class by a single element. 

Let $\cM=(E,\cI)$ be a matroid and $E',E''\subseteq E$. The \emph{restriction to} $E'$ and the \emph{deletion of} $E\setminus E'$ result in the same matroid $\cM|E'=\cM\setminus(E\setminus E')=(E',\cI')$ with independence family $\cI'=\{I\in\cI\colon I\subseteq E'\}$. The \emph{contraction to} $E''$ and the \emph{contraction of} $E\setminus E''$ result in the same matroid $\cM. E''=\cM/(E\setminus E'')=(E'',\cI'')$ where $\cI''=\{I\in\cI\colon I\subseteq E'',I\cup Z\in\cI\text{ for any }Z\in\cI\text{ with }Z\subseteq E\setminus E''\}$. A matroid $\cN$ that can be obtained from $\cM$ by a sequence of restrictions and contractions is called a \emph{minor} of $\cM$. 

A matroid is \emph{representable over some field} $\mathbb{F}$ if there exists a family of vectors from a vector space over $\mathbb{F}$ whose linear independence relation is the same as the independence relation of the matroid. The matroid is \emph{binary} if it is representable over $GF(2)$, and is \emph{regular} if it can be represented over any field. The \emph{complete binary matroid} of rank $r$ is the matroid
$\cB_r$ represented over $\mathbb F_2$ by all nonzero vectors
of $\mathbb F_2^r$. The following lemma gives a characterization of binary matroids in terms of cycles, see e.g.~\cite[Theorem~9.1.2]{oxleyMatroid}.

\begin{thm} \label{thm:bin}
A matroid is binary if and only if $C_1 \triangle C_2$ is a cycle for any cycles $C_1,C_2$.
\end{thm}

The matroid $R_{10}$ is a binary matroid that can be represented as the ten vectors in the five-dimensional vector space over $GF(2)$ that have exactly three nonzero entries. The \emph{Fano matroid} $F_7$ is obtained from the Fano plane by calling a set independent if it contains at most two points or it has three points which are not lines of the plane. In other words, $F_7$ is the matroid with ground set $E=\{a,b,c,d,e,f,g\}$ whose bases are all subsets of size $3$ except $\{a,b,d\}$, $\{b,c,e\}$, $\{a,c,f\}$, $\{a,e,g\}$, $\{c,d,g\}$, $\{b,f,g\}$ and $\{d,e,f\}$.

The \emph{uniform matroid} $\cM=(E,\cI)$ of rank $r$ is defined as $\cI=\{I\subseteq E\colon |I|\leq r\}$. For a graph $G=(V,E)$, the \emph{graphic matroid} $\cM=(E,\cI)$ of $G$ is defined on the edge set by considering a subset $F\subseteq E$ to be independent if it is a forest, that is, $\cI=\{F\subseteq E\colon F\ \text{does not contain a cycle}\}$. 
A matroid $\cM=(E,\cI)$ of rank $r(\cM)$ is called a \emph{paving matroid} if every circuit of $\cM$ has size at least $r(\cM)$. A paving matroid is \emph{sparse} if its dual is also paving. The structure of paving and sparse paving matroids is described by the following theorem~\cite{hartmanis1959lattice,welsh2010matroid,frank2011connections}.

\begin{thm}[Hartmanis~\cite{hartmanis1959lattice}] \label{thm:pav} 
Let $r\ge 2$, and let $\mathcal H$ be a possibly empty family of proper subsets of a set $E$ with $|E|\ge r$. Suppose that every $H\in\mathcal H$ has size at least $r$, and that $|H\cap H'|\le r-2$ for all distinct $H,H'\in\mathcal H$. Define 
\[ \mathcal B_{\mathcal H} = \bigl\{B\in\textstyle\binom{E}{r}\colon B\nsubseteq H\text{ for every }H\in\mathcal H\bigr\}. \] 
Then $\mathcal B_{\mathcal H}$ is the set of bases of a rank-$r$ paving matroid on $E$ whose circuit-hyperplanes are exactly the size-$r$ hyperedges of $\cH$, which is sparse paving exactly if $|H|=r$ for each $H\in\mathcal{H}$. Conversely, every paving and sparse paving matroid arises in this way. 
\end{thm}

Two matroids $\cM_1$ and $\cM_2$ are \emph{isomorphic}, denoted by $M_1\cong M_2$, if there is a bijection $\psi\colon E(\cM_1)\to E(\cM_2)$ such that $I\subseteq E(\cM_1)$ is independent in $\cM_1$ if and only if $\psi(I)$ is independent in $\cM_2$.

\paragraph{Seymour's decomposition.}

Let $\cM_1$ and $\cM_2$ be binary matroids on ground sets $E_1$ and $E_2$, respectively, such that $|E_1|,|E_2|<|E_1\triangle E_2|$. Then, we denote by $\cM_1\triangle\cM_2$ the binary matroid on ground set $E=E_1\triangle E_2$ with cycles being the sets of the form $C_1\triangle C_2$ where $C_i$ is a cycle of $\cM_i$ for $i=1,2$ with $C_1\cap E_1\cap E_2 = C_2\cap E_1\cap E_2$.

When $E_1\cap E_2=\emptyset$, then $\cM_1\oplus_1\cM_2\coloneqq \cM_1\triangle\cM_2$ is called the \emph{1-sum} or \emph{direct sum} of $\cM_1$ and $\cM_2$. Its family of bases is
$$
\cB(\cM_1\oplus_1\cM_2)=\{B_1\cup B_2\colon B_1\in\cB(\cM_1),B_2\in\cB(\cM_2)\}.
$$

When $|E_1\cap E_2|=1$, say $E_1\cap E_2=\{t\}$, such that $t$ is neither a loop nor a coloop of $\cM_1$ or $\cM_2$, then $\cM_1\oplus_2\cM_2\coloneqq \cM_1\triangle\cM_2$ is called the \emph{2-sum} of $\cM_1$ and $\cM_2$ along $t$. Its family of bases is
$$
\begin{aligned}
\cB(\cM_1\oplus_2\cM_2)
=&\{B_1'\cup B_2\colon B_1'\in\cB(\cM_1/t),B_2\in\cB(\cM_2\backslash t)\}\\
&\cup\{B_1\cup B_2'\colon B_1\in\cB(\cM_1\backslash t),B_2'\in\cB(\cM_2/t)\}.
\end{aligned}
$$

When $|E_1\cap E_2|=3$ and $E_1\cap E_2=T$ is a coindependent triangle of both $\cM_1$ and $\cM_2$, then $\cM_1\oplus_3\cM_2\coloneqq \cM_1\triangle\cM_2$ is called the \emph{3-sum} of $\cM_1$ and $\cM_2$ along $T$. The matroid $\cM_1\oplus_3\cM_2$ has rank $r(\cM_1)+r(\cM_2)-2$, and its family of bases is
$$
\begin{aligned}
\cB(\cM_1\oplus_3\cM_2)
=&\{B_1''\cup B_2\colon B_1''\in\cB(\cM_1/T),B_2\in\cB(\cM_2\backslash T)\}\\
&\cup\{B_1'\cup B_2'\colon B_1'\subseteq E_1\backslash T,B_2'\subseteq E_2\backslash T,\exists i,j,k\colon \{i,j,k\}=\{1,2,3\},\\
&\qquad B_1'+t_i,B_1'+t_j\in\cB(\cM_1),B_2'+t_i,B_2'+t_k\in\cB(\cM_2)\}\\
&\cup\{B_1\cup B_2''\colon B_1\in\cB(\cM_1\backslash T),B_2''\in\cB(\cM_2/T)\}.
\end{aligned}
$$
Seymour's decomposition theorem gives a constructive characterization of regular matroids.

\begin{thm}[Seymour~\cite{seymour1980decomposition}]\label{thm:seym}
    A matroid is regular if and only if it is obtained by means of 1-, 2-, and 3-sums, starting from graphic and cographic matroids and copies of $R_{10}$.
\end{thm}

A matroid $\cM$ is \emph{connected}, or \emph{$2$-connected}, if it is not a $1$-sum of two matroids with nonempty ground sets. It is said to be \emph{3-connected} if it is neither a $1$-sum nor a $2$-sum of two matroids. The following two theorems give a characterization of 2-connected and 3-connected regular matroids.

\begin{thm}[Theorem 8.3.1 of ~\cite{oxleyMatroid}]\label{thm:seym-2}
    A 2-connected matroid $\cM$ is not 3-connected if and only if $\cM = \cM_1 \oplus_2 \cM_2$ for some matroids $\cM_1$ and $\cM_2$, each of which has at least three elements and is isomorphic to a proper minor of $\cM$.
\end{thm}

\begin{thm}[Corollary 13.4.6 of~\cite{oxleyMatroid}]
\label{thm:seym-3}
Let $\cM$ be a 3-connected regular matroid. Then at least one of the following alternatives holds:
\begin{enumerate}[label=(\roman*)]\itemsep=0em
    \item \label{it:i}
    $\cM$ is graphic.
    \item \label{it:ii}
    $\cM$ is cographic.
    \item \label{it:iii}
    $\cM \cong R_{10}$.
    \item \label{it:iv}
    There are regular matroids $\cM_1$ and $\cM_2$ such that
    $E(\cM_1)\cap E(\cM_2)=T$, where $T$ is a triangle of both
    $\cM_1$ and $\cM_2$, and
    $\cM=\cM_1\oplus_3\cM_2$. In addition, for each $i\in\{1,2\}$,
    the following properties hold:
    \begin{enumerate}[label=(\alph*)]\itemsep=0em
        \item \label{it:a}
        $\cM_i$ is internally 3-connected, and every 2-element
        2-separating set of $\cM_i$ meets $T$.
        \item \label{it:b}
        $\cM_i$ is isomorphic to a minor of $\cM$.
        \item \label{it:c}
        $|E(\cM_i)\setminus \clo_{\cM_i}(T)|\ge 6$ and
        $|E(\si(\cM_i))|\ge 9$.
    \end{enumerate}
\end{enumerate}
\end{thm}

\paragraph{Ear decomposition of matroids.}

Let $\cM$ be a matroid on ground set $E$, and let
$C_1,\dots,C_t$ be a nonempty sequence of circuits. Set
$D_0=\varnothing$ and $D_i=\bigcup_{j=1}^i C_j$. The sequence is a
\emph{partial ear decomposition} if, for every $i\ge2$,
\begin{enumerate}[label=(E\arabic*)]\itemsep0em
        \item $C_i\cap D_{i-1}\ne\varnothing$, \label{it:e1}
        \item $C_i\setminus D_{i-1}\ne\varnothing$, \label{it:e2}
        \item no circuit satisfying \ref{it:e1} and \ref{it:e2}
        has its new elements properly contained in
        $C_i\setminus D_{i-1}$. \label{it:e3}
\end{enumerate}
We call the sequence an \emph{ear decomposition} if $D_t=E$. The set
$P_i=C_i\setminus D_{i-1}$ is the $i$th \emph{lobe}. A lobe is odd or
even according to its cardinality.

It is known that a matroid is connected if and only if any two distinct elements of $\cM$ are contained in a common circuit; this condition is vacuous for one-element matroids. Therefore, the one-element rank-$1$ uniform matroid $U_{1,1}$ is the only connected matroid without an ear decomposition. This exceptional case is the reason why, in statements involving ear decompositions or the parameter $\eta$, we must exclude $U_{1,1}$.

We will use the following lemma, which is a matroidal analogue of the well-known fact that the graphic matroid of a graph is connected if and only if the graph has an ear decomposition.

\begin{thm}[Coullard and Hellerstein~\cite{coullard1996independence}]
\label{thm:ear}
Let $\cM$ be a matroid not isomorphic to $U_{1,1}$. Then $\cM$ has an ear decomposition if and only if it is connected. Moreover, if $\cM$ is connected, then every partial ear decomposition of $\cM$ extends to an ear decomposition.
\end{thm} 

We give some basic properties on ear decompositions of matroids.

\begin{lem}\label{lem:lobe-contained}
Let $\cM$ be a matroid with an ear decomposition $C_1,\dots,C_t$ and let $D_i=\bigcup_{j=1}^i C_j$ and $P_i=C_i\setminus D_{i-1}$ for $i\in [t]$.  
If $C\subseteq D_i$ is a circuit not contained in $D_{i-1}$, then
$P_i\subseteq C$.
\end{lem}

\begin{proof}
If $i=1$, then $C\subseteq C_1$ and circuit minimality gives
$C=C_1=P_1$. Assume $i\ge2$. The circuit $C$ satisfies
\ref{it:e1} and \ref{it:e2} at step $i$. Moreover,
$C\setminus D_{i-1}=C\cap P_i\subseteq P_i$. The minimality assumption
\ref{it:e3} therefore forces $C\cap P_i=P_i$.
\end{proof}

\begin{lem}\label{lem:number-of-ears}
Every ear decomposition of a matroid $\cM$ has
$|E(\cM)|-r(\cM)$ ears, and its lobes satisfy
$$
        r(\cM)=\sum_i(|P_i|-1).
$$
\end{lem}

\begin{proof}
Condition \ref{it:e3} says precisely that $P_i$ is a circuit of the
contraction $\cM/D_{i-1}$. Hence
$$
 r(D_i)-r(D_{i-1})
 =r_{\cM/D_{i-1}}(P_i)
 =|P_i|-1.
$$
Summing over all ears yields
$$
 r(\cM)=\sum_i(|P_i|-1)=|E(\cM)|-t,
$$
because the lobes partition $E(\cM)$.
\end{proof}

Let $\varphi(\cM)$ denote the minimum number of even lobes in an ear
decomposition. Since $\eta(\cM)=(r(\cM)+\varphi(\cM))/2$, and each even lobe contributes at least $1$ to $r(\cM)=\sum_i (|P_i|-1)$,
we have $\varphi(\cM)\le r(\cM)$. \Cref{lem:number-of-ears} and parity give
\begin{equation}\label{eq:eta-rank-phi}
        \eta(\cM)=\frac{r(\cM)+\varphi(\cM)}2\le r(\cM).
\end{equation}
Note that every join is independent: if a join contained a circuit $C$, it
would meet $C$ in more than $|C|/2$ elements. Consequently,
\begin{equation}\label{eq:mu-at-most-rank}
        \mu(\cM)\le r(\cM).
\end{equation}

Frank showed the following.

\begin{thm}[Frank~\cite{frankConservative}]\label{thm:frank}
For a connected graphic matroid $\cM$ not isomorphic to $U_{1,1}$, we have $\mu(\cM)=\eta(\cM)$.
\end{thm}

\paragraph{Covering radius and maximum frustration.}

For a graph $G$, the \emph{cutset code} of $G$ is the binary linear code $\cB(G)=\spa_{\mathbb F_2}\{\delta(U)\colon U\subseteq V(G)\}\subseteq \mathbb F_2^{E(G)}$, where $\spa_{\mathbb F_2}$ denotes linear span over $\mathbb F_2$, with cuts identified with their incidence vectors in $\mathbb F_2^{E(G)}$.
Equivalently, $\cB(G)$ is the cut space of $G$ over $\mathbb F_2$.

The \emph{covering radius} of a binary code $\cB\subseteq\mathbb F_2^E$ is $\rho(\cB)=\max_{x\in\mathbb F_2^E}\min_{b\in\cB}|x+b|$, where addition is taken coordinatewise over $\mathbb F_2$ and $|\cdot|$ denotes the Hamming weight.

A signing of $G$ is a map $\sigma\colon E(G)\to\{0,1\}$, where edges $e$ with $\sigma(e)=1$ are called \emph{negative} and edges with $\sigma(e)=0$ are called \emph{positive}. We write
$N_\sigma=\{e\in E(G):\sigma(e)=1\}$
for the set of negative edges. A signing $\sigma$ is \emph{balanced} if every circuit of $G$ contains an even number of negative edges.

Following Bowlin~\cite{Bowlin2012}, the \emph{frustration index} $f(G,\sigma)$ is the minimum number of edges whose deletion makes the signing balanced. Equivalently, $f(G,\sigma)$ is the minimum possible number of negative edges obtainable from $\sigma$ by \emph{switching}, swapping negative and positive edges across a cut in the graph. Two signings $\sigma,\sigma'$ are said to be switching-equivalent, $\sigma'\sim \sigma$ if they are obtained by switching across a cut in the graph. We have that,
\[
f(G,\sigma)
=
\min_{\sigma'\sim \sigma}
|N_{\sigma'}|,
\]
The \emph{maximum frustration} of $G$ is
\[
F_{\max}(G)
=
\max_{\sigma:E(G)\to\{0,1\}}
f(G,\sigma).
\]
Sol\'e and Zaslavsky~\cite{SoleZaslavsky1994} showed that
$F_{\max}(G)=\rho(\cB(G))$.

Switching-equivalent is an equivalence relation and the equivalence classes are called switching classes. We call a signing $\sigma$ \emph{reduced} if it has the smallest possible number of negative edges among the switching class containing $\sigma$. Equivalently, $\sigma$ is reduced if switching by any vertex set $U\subseteq V(G)$ does not decrease the number of negative edges.

\section{Limits of the Min-Max Formula for General Matroids}
\label{sec:limits}

Frank's theorem gives a constructive polynomial-time equality for graphic matroids: the join parameter and the ear-decomposition parameter are equal, and the corresponding optima can be found efficiently. For matroids representable over fields of characteristic $2$, the ear-decomposition side remains algorithmically tractable: the work of Szegedy and Szegedy~\cite{Szegedy2006}, together with the algorithmic discussion in Szegedy's thesis~\cite{SzegedyThesis2005}, yields a randomized polynomial-time algorithm for computing $\varphi(\cM)$ from such a representation. 

The join side, however, has a different character. For a binary matroid, $\mu(\cM)$ is the covering radius of its cycle space, and Guruswami, Micciancio, and Regev~\cite{GuruswamiMicciancioRegev2005} showed that the covering-radius problem for linear codes is NP-hard to approximate within any fixed constant factor and $\Pi_2$-hard for some fixed constant factor. In this section we show that related hardness applies for cographic and paving matroids.

\subsection{Failure of Equality in Cographic Matroids}
\label{sec:noeq}

We begin with examples showing that the failure of Frank's equality is not merely a matter of one parameter always dominating the other. The cographic example below gives $\mu(\cM)<\eta(\cM)$.

\begin{prop}
\label{prop:k44-example}
The matroid $\cM=\cM^*(K_{4,4})$ satisfies $\mu(\cM)=4$, $\gamma(\cM)=6$, and $\eta(\cM)=5$. In particular, $\mu(\cM)<\eta(\cM)$.
\end{prop}

\begin{proof}
Let $A$ and $B$ be the two vertex classes of $K_{4,4}$. If a set
$X\subseteq V(K_{4,4})$ contains $x$ vertices from $A$ and $y$ vertices
from $B$, then
\begin{equation*}
        |\delta(X)|=x(4-y)+y(4-x)=4(x+y)-2xy.
\end{equation*}
In particular, every bond of $K_{4,4}$ has even size. The bonds of size
$4$ are the vertex bonds, the bonds of size $6$ are precisely the cuts
$\delta(\{u,v\})$ with $u\in A$ and $v\in B$, and all other bonds have
size at least $8$.

We first show that $\mu(\cM)=4$. Any perfect matching of $K_{4,4}$ is a
join: it meets each vertex bond in one edge, each bond of the form
$\delta(\{u,v\})$ with $u\in A$ and $v\in B$ in either zero or two edges,
and every other bond has size at least $8$.

It remains to show that no join has size at least $5$. Suppose, to the
contrary, that $J$ is a join with $|J|\ge 5$, and choose
$J_0\subseteq J$ with $|J_0|=5$. Since each vertex bond has size $4$,
every vertex is incident with at most two edges of $J_0$. Let $A_2$ and
$B_2$ be the sets of degree-$2$ vertices of the graph
$(V(K_{4,4}),J_0)$ in the two bipartition classes. Both sets are nonempty.
If some $u\in A_2$ and $v\in B_2$ are nonadjacent in $J_0$, then the bond
$\delta(\{u,v\})$ has size $6$ and contains four edges of $J_0$, a
contradiction.

Thus every vertex of $A_2$ is adjacent in $J_0$ to every vertex of $B_2$.
Since all degrees in $J_0$ are at most $2$, we have $|A_2|\le 2$ and
$|B_2|\le 2$. We now distinguish the possible cases.

If $|A_2|=2$ and $|B_2|=2$, then the four edges between $A_2$ and $B_2$
belong to $J_0$. Let $cq$ be the fifth edge of $J_0$, with $q\in B$.
Then $q\notin B_2$, and the bond $\delta(A_2\cup\{q\})$ has size $8$ and
contains all five edges of $J_0$, a contradiction.

If $|A_2|=2$ and $|B_2|=1$, write $A_2=\{a,b\}$ and $B_2=\{p\}$. Besides
the edges $ap$ and $bp$, each of $a$ and $b$ is incident with one further
edge of $J_0$. Let $cq$ be the remaining edge of $J_0$, where
$q\in B\setminus B_2$. Then the bond $\delta(A_2\cup\{q\})$ has size $8$
and contains all five edges of $J_0$, again a contradiction. The case
$|A_2|=1$ and $|B_2|=2$ is symmetric.

Finally, suppose that $A_2=\{a\}$ and $B_2=\{p\}$. Write the two edges
incident with $a$ as $ap$ and $aq$, and write the other edge incident with
$p$ as $bp$. The two remaining edges of $J_0$ have the form $cr$ and $ds$,
where $r,s\in B\setminus\{p,q\}$. Then the bond
$\delta(\{a,b,r,s\})$ has size $8$ and contains all five edges of $J_0$,
a contradiction. Hence every join has size at most $4$, and therefore
$\mu(\cM)=4$.

Since $K_{4,4}$ has $16$ edges and $8$ vertices, the cographic matroid
$\cM$ has $16$ elements and rank $16-(8-1)=9$. Hence every ear
decomposition has seven lobes by \Cref{lem:number-of-ears}. As
observed above, every bond of $K_{4,4}$ has even size. Thus the first
lobe of every ear decomposition of $\cM$ is even, and so
$\gamma(\cM)\le 6$.

Now label the two vertex classes by $\{a,b,c,d\}$ and $\{1,2,3,4\}$.
The vertex bonds
\begin{equation*}
        \delta(a),\ \delta(1),\ \delta(2),\ \delta(3),\
        \delta(b),\ \delta(c),\ \delta(d)
\end{equation*}
form an ear decomposition of $\cM$, with lobes
\begin{equation*}
        \{a1,a2,a3,a4\},\ \{b1,c1,d1\},\ \{b2,c2,d2\},\
        \{b3,c3,d3\},\ \{b4\},\ \{c4\},\ \{d4\}.
\end{equation*}
This ear decomposition has six odd lobes. Therefore $\gamma(\cM)=6$, and
$\eta(\cM)=(16-6)/2=5$.
\end{proof}

\begin{rem}\label{rem:minor}
To discuss minor-closedness, we first extend the parameters naturally to arbitrary matroids. If the connected components of a matroid $\cM$ are $\cM_1,\dots,\cM_k$, define
\[
\mu(\cM)=\sum_{i=1}^k \mu(\cM_i)
\qquad\text{and}\qquad
\eta(\cM)=\sum_{i=1}^k \eta(\cM_i),
\]
with the convention that a $U_{1,1}$ component contributes $1$ to both parameters. This definition of $\mu$ agrees with the join definition, since the circuits of a direct sum are precisely the circuits of its components.

Under this extension, the class of matroids satisfying the join--ear equality is not minor-closed. Indeed, let $G$ be obtained from $K_{4,4}$ by adding one edge between two vertices in the same bipartition class, and let $\cN=\cM^*(G)$. A direct computation gives $\mu(\cN)=\eta(\cN)=5$. However, $\cM^*(K_{4,4})$ is a minor of $\cN$, and \Cref{prop:k44-example} shows that it does not satisfy $\mu=\eta$. The same proposition also implies that $\cM^*(K_{4,4})$ is not graphic. Indeed, every connected graphic matroid satisfies the join--ear equality by Frank's result~\cite{frankConservative}, whereas $\cM^*(K_{4,4})$ does not.
\end{rem}

\subsection{Incomparability}
\label{sec:incomp}

The previous example shows that the graphic join--ear equality may fail with $\mu(\cM)<\eta(\cM)$. The situation is even worse: the reverse inequality can also occur. Thus, for general matroids, the parameters $\mu$ and $\eta$ are incomparable.

\begin{prop}
\label{prop:incomparable}
There exists a connected paving matroid $\cM$ such that $\mu(\cM)=3$ and $\eta(\cM)=2$. In particular, $\eta(\cM)<\mu(\cM)$.
\end{prop}

\begin{proof}
Let $\cM$ be the rank-$4$ ternary matroid represented over $\mathbb F_3$ by
\begin{equation*}
A=
\begin{pmatrix}
1&1&1&1&1&0&0\\
0&1&1&0&1&1&1\\
1&1&0&2&2&1&0\\
0&1&1&2&0&1&2
\end{pmatrix},
\end{equation*}
where the columns are labelled $1,\dots,7$. We write, for instance, $1236$ for the set ${1,2,3,6}$. The circuits of $\cM$ are
\begin{equation*}
1236,\ 1245,\ 1357,\ 1467,\ 2347,\ 2567,\ 13456,\ 23456,\ 34567.
\end{equation*}
In particular, $\cM$ is paving. Moreover, the sequence
\begin{equation*}
C_1=13456,\qquad C_2=1245,\qquad C_3=2347
\end{equation*}
is an ear decomposition of $\cM$, with lobes $13456$, $2$, and $7$. Hence $\cM$ is connected. Since $\cM$ has rank $4$ and contains a spanning circuit, \Cref{lem:paving-eta-formula} gives $\eta(\cM)=2$.

Now let $J={1,2,7}$. From the above list, $J$ intersects each $4$-circuit in exactly two elements and each $5$-circuit in exactly one element. Thus $J$ is a join, and so $\mu(\cM)\geq 3>2=\eta(\cM)$. To determine $\mu(\cM)$ exactly, it remains to show that no join has size at least $4$. The following table assigns to every $3$-set different from $127$ a circuit containing it:
\begin{equation*}
\begin{array}{ll}
1236: & 123,126,136,236;\\
1245: & 124,125,145,245;\\
1357: & 135,137,157,357;\\
1467: & 146,147,167,467;\\
2347: & 234,237,247,347;\\
2567: & 256,257,267,567;\\
13456: & 134,156,345,346,356,456;\\
23456: & 235,246;\\
34567: & 367,457.
\end{array}
\end{equation*}
Each circuit in the table has size $4$ or $5$, so each listed $3$-set meets its assigned circuit in more than half of its elements. Hence no $3$-set other than $127$ is a join. But every set of size at least $4$ contains a $3$-subset different from $127$, and therefore no join has size at least $4$. Thus $\mu(\cM)=3$, and consequently $\eta(\cM)=2<3=\mu(\cM)$.
\end{proof}

\subsection{Hardness of Maximum Cographic Matroid Join}
\label{sec:cographicHardness}

We prove that computing the maximum join of a cographic matroid is NP-hard. The proof uses the equivalence between joins of cographic matroids, maximum frustration of signed graphs, and the covering radius of cutset codes. We first establish this equivalence, then reduce from $\textsc{Maximum Stable Set}$ in cubic graphs.

For the reduction, we will use the following equivalent formulation of maximum frustration. Let $\oplus$ denote addition modulo $2$. For a graph $G$, a signing $\sigma\colon E(G)\to\{0,1\}$, and a switching vector $x\in\{0,1\}^{|V(G)|}$, we have
\[
f(G,\sigma)
=
\min_{x\in\{0,1\}^{V(G)}}
\sum_{uv\in E(G)}
\bigl(\sigma(uv)\oplus x_u\oplus x_v\bigr).
\]
The summand
$
\sigma(uv)\oplus x_u\oplus x_v
$
is equal to $1$ exactly when the edge $uv$ is negative after switching by $x$. Consequently,
\[
F_{\max}(G)
=
\max_{\sigma:E(G)\to\{0,1\}} f(G,\sigma).
\]
\subsubsection{Cographic Joins and Maximum Frustration}

Now let $\cM=M^*(G)$ be the cographic matroid of the graph $G$. We explain why the join parameter of $\cM$ is exactly the maximum frustration of $G$.

\begin{lem}
    $\mu(M^*(G))=F_{\max}(G)=\rho(\cB(G)).$
\end{lem}

\begin{proof}
    Recall that the circuits of $M^*(G)$ are the bonds of $G$. Hence a set $J\subseteq E(G)$ is a join of $\cM$ if and only if
$
|J\cap B|\le |B|/2
$
for every bond $B$ of $G$. Equivalently, this inequality holds for every cut $\delta(U)$, since every cut is a disjoint union of bonds.

Given $J\subseteq E(G)$, let $\sigma_J:E(G)\to\{0,1\}$ be the signing whose negative-edge set is $J$; that is, $\sigma_J(e)=1$ if and only if $e\in J$. Switching $\sigma_J$ by a vertex set $U\subseteq V(G)$ changes the negative-edge set from $J$ to $J\triangle\delta(U)$. Therefore
$
|J\triangle\delta(U)|
=
|J|+|\delta(U)|-2|J\cap\delta(U)|.
$
It follows that switching by $U$ does not decrease the number of negative edges if and only if
$
|J\cap\delta(U)|\le |\delta(U)|/2.
$
Thus $J$ is a join of $M^*(G)$ if and only if the signing $\sigma_J$ is reduced.

If $\sigma_J$ is reduced, then it already has the minimum possible number of negative edges in its switching class. Hence
$
f(G,\sigma_J)=|J|.
$
Taking the maximum over all joins $J$ gives
$
\mu(M^*(G))=F_{\max}(G).
$
Together with the covering-radius interpretation from the preliminaries, we obtain
$
\mu(M^*(G))=F_{\max}(G)=\rho(\cB(G)).
$
\end{proof}

\subsubsection{A cubic-graph estimate}

We prove the following elementary lemmas about cubic-graphs, in which every vertex has degree three, to be used in proving the correctness of the NP-hardness reduction. Let $G=(V,E)$ be a cubic graph, and let $\alpha(G)$ denote its independence number.

\begin{lem}\label{lem:switchingSet}
    For every $N\subseteq E$, there exists a set $Y \subset V$ such that $2|Y|+|N\Delta \delta_G(Y)| \leq |E|-\alpha(G)$.
\end{lem}

\begin{proof}
    Let $I\subseteq V$ be a maximum stable set, so $|I|=\alpha(G)$. Define $Y=\{v\in I\colon \delta_G(v)\subseteq N\}$. For every $v\in I\setminus Y$, the star $\delta_G(v)$ contains an edge of $E\setminus N$. These edges are distinct, because $I$ is stable, and so the stars $\delta_G(v)$ for $v\in I$ are pairwise edge-disjoint. Hence $|I\setminus Y|\le |E\setminus N|=|E|-|N|$, and therefore $|Y|\ge \alpha(G)-(|E|-|N|)$. Note that $\delta_G(Y)\subseteq N$ and $|\delta_G(Y)|=3|Y|$. Thus
$2|Y|+|N\triangle \delta_G(Y)|=2|Y|+(|N|-3|Y|)=|N|-|Y|\le |E|-\alpha(G)$.
\end{proof}

\begin{lem}\label{lem:stableSetIsGood}
    For every $Y\subseteq V$, we have $2|Y|+|E\Delta \delta_G(Y)| \geq |E|-\alpha(G)$. Moreover, there exists a set $Y^*$ that attains the lower bound.
\end{lem}

\begin{proof}
    For any $Y\subseteq V$, $|E\triangle \delta_G(Y)|=|E\setminus \delta_G(Y)|=|E|-|\delta_G(Y)|$. Since $G$ is cubic, $|\delta_G(Y)|=3|Y|-2e_G(Y)$, where $e_G(Y)$ is the number of edges of $G[Y]$. Therefore $2|Y|+|E\triangle \delta_G(Y)|=|E|-|Y|+2e_G(Y)$. Since $G[Y]$ has a vertex cover of size at most $2e_G(Y)$, it has a stable set of size at least $|Y|-2e_G(Y)$ by \Cref{thm:gallai}. Thus $|Y|-2e_G(Y)\le \alpha(G)$ for every $Y\subseteq V$, which implies $2|Y|+|E\triangle \delta_G(Y)|\ge |E|-\alpha(G)$. Therefore equality holds at any maximum stable set $Y^*=I$ of $G$.
\end{proof}

\begin{cor}\label{lem:cubic-indep}
For every $N\subseteq E$,
\begin{equation*}
\min_{Y\subseteq V} \{2|Y|+|N\triangle \delta_G(Y)|\}\le |E|-\alpha(G).
\end{equation*}
Moreover, equality in the above statement holds for $N=E$.
\end{cor}

\begin{proof}
    Follows from the above two lemmas.
\end{proof}

\subsubsection{Gadget properties}

We now discuss the gadget estimates needed for the reduction. The statements are simple and will be used directly in the proof of the hardness theorem; their verification is a somewhat technical case analysis. Readers interested mainly in the structure of the reduction may first read the statements, and then continue with \Cref{sec:stable-set-reduction}.

The gadget we will be using is $K_{3,3}$ with an extra edge. For the rest of this reduction, we will use $H$ to denote the gadget graph and fix the labeling of the vertices and the placement of the extra edge as described in \Cref{fig:gadget}.

\begin{figure}[th!]
\centering
\begin{subfigure}[t]{0.5\textwidth}
    \centering
    \includegraphics[scale=0.6]{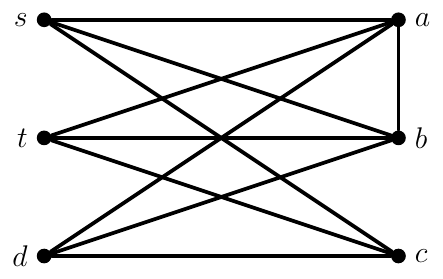}
    \caption{The gadget graph $H$, $K_{3,3}$ with an extra edge.}
    \label{fig:gadget}
\end{subfigure}\hfill
\begin{subfigure}[t]{0.5\textwidth}
    \centering
    \includegraphics[scale=0.6]{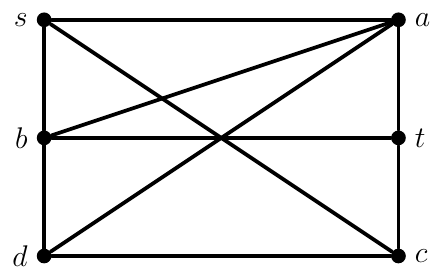}
    \caption{Principal cut swapped along $tb$ does not contain edges $sb,bd,at,tc$, but contains $ab$}
    \label{fig:swap}
\end{subfigure}\hfill
\caption{Illustration of the swapping of principal cut along the edge $tb$ in gadget.}
\label{fig:swapOperation}
\end{figure}
Let $L=\{s,t,d\}$, $R=\{a,b,c\}$ and $V(H)=L\uplus R$.
We call $\delta(L)$ the \emph{principal} cut. 

For an edge $uv\neq ab$ with $u\in L$ and $v\in R$, we say that the principal cut is \emph{swapped along} $uv$ if $u$ and $v$ exchange sides. The resulting cut is $\delta(L\triangle\{u,v\})$. For example, swapping along $tb$ gives the cut $\delta({s,b,d})$, shown in \Cref{fig:swap}.

Note that the operation of swapping removes exactly 4 edges from the principal cut and adds the extra edge $ab$. This operation is very useful in proving combinatorial properties about the gadget. 

For a graph $G$, with a signing $\sigma:E(G)\to \{0,1\}$ and a switching $x\in \{0,1\}^{|V(G)|}$ define
\[
h(G,\sigma,x) \coloneqq
\sum_{uv\in E(G)} \bigl(\sigma(uv)\oplus x_u\oplus x_v\bigr)
\]
as the number of negative edges after switching along the cut corresponding to the set with characteristic vector $x$. We can assume $x_s=0$ without loss of generality as switching by a set is the same as switching along its complement. Conditioning on the value of $x_t$ for vertex $t$ in the gadget $H$ in the switching vector we define
\[
f_\beta(H,\sigma) \coloneqq \min_{\substack{x_s=0, x_t=\beta \\ x_a,x_b,x_c,x_d\in\{0,1\}}} h(H,\sigma,x) 
\]
for $\beta\in\{0,1\}$.

Thus $f_0(H,\sigma)$ is the minimum number of negative edges among switchings that keep $s$ and $t$ on the same side, while $f_1(H,\sigma)$ is the minimum among switchings that keep $s$ and $t$ on opposite sides. 

Let $\delta^-_\sigma(S)$ denote the number of negative edges in the signing $\sigma$ in the cut corresponding to the set $S$. In the following proofs, successive switchings by sets $S_1,S_2,\ldots$ are equivalent to a single switching by $S_1\triangle S_2\triangle\cdots$. We therefore describe switchings sequentially without explicitly tracking the final switching vector.

\begin{lem}\label{lem:minis2}
    For every signing $\sigma$ of the gadget $H$, we have $\min(f_0(H,\sigma),f_1(H,\sigma))\leq 2$.
\end{lem}

\begin{proof}
    This is equivalent to proving that $f(H,\sigma)\leq 2$ that is there exists a reduced signing with at most $2$ negative edges.
    
    Given any signing $\sigma$ of the gadget graph $H$, one can ensure that at most $1$ negative edge is incident to every vertex in $L$ by switching along the degree cuts of $s,t,d$ independently, if the number of negative edges incident to the corresponding vertex is at least $2$. Let $\gamma$ be such a signing. So we have $\delta^-_\gamma(s)\leq 1,\delta^-_\gamma(t)\leq 1,\delta^-_\gamma(d)\leq 1$.

    \paragraph{Case $\sigma(ab)=0$:} If at least one of $\delta^-_\gamma(s),\delta^-_\gamma(t),\delta^-_\gamma(d)$ is zero we conclude as $\gamma$ is the switching equivalent signing of $\sigma$ with at most two negative edges.

    If all of $\delta^-_\gamma(s),\delta^-_\gamma(t),\delta^-_\gamma(d)$ are equal to $1$, and all of them incident to a single vertex $v \in R$, then we conclude after switching along the corresponding degree cut leading to at most one negative edge (if $v$ is $a$ or $b$).

    If all of $\delta^-_\gamma(s),\delta^-_\gamma(t),\delta^-_\gamma(d)$ are equal to $1$, and all of them are incident to two vertices say $u,v\in R$ with one negative edge incident to $u$ say $xu$ and two negative edges incident to $v$ say $yv,zv$. If $v=c$, then switching along the degree cut of $c$, decreases a negative edge, giving us a signing with at most two negative edges. If $v=a$ or $b$, then switching along the degree cut of $v$ makes the edge $ab$ negative along with $xv$. Since two negative edges are incident to vertex $x\in L$ switching along that reduces the negative edges by $1$ and we have a switching with at most two negative edges.

    If all of them are equal to $1$, and all of them are incident to three vertices forming a matching, say $xa,yb,zc$. Let $S$ be the cut obtained by swapping the principal cut along the matched edge of $c$, $zc$. The cut $S$ contains only $5$ edges out of which there are $3$ negative (matching) edges, thus switching along the cut $S$ gives a signing with at most two negative edges.

    \paragraph{Case $\sigma(ab)=1$:} If at most one of $\delta^-_\gamma(s),\delta^-_\gamma(t),\delta^-_\gamma(d)$ is one, we conclude as $\gamma$ is the switching equivalent signing of $\sigma$ with at most two negative edges.

    If exactly two of them are $1$ and both the negative edges are incident to a vertex $v\in R$, then switching along the degree cut of $v$ reduces the number of negative edges by $1$ if $v=c$ and reduces by $2$ if $v=a$ or $b$.

    If exactly two of them are $1$ and both the negative edges are incident to two different vertices $u,v\in R$ say $xu,yv$ with $x,y\in L$. At least one of $u,v$ has to be $a$ or $b$. Let $v$ be such vertex. Switching along the degree cut of $v$, changes $xv,zv$ to negative, and $yv,ab$ to positive where $\{z\}=L\setminus \{x,y\}$. The negative edges are $xu,xv,zv$. Since $x$ now has two negative edges incident to it, switching along the degree cut of $x$ reduces the number of negative edges by $1$.

    If all of them are equal to $1$, and all of them incident to a single vertex $v \in R$, then switching along the degree cut of $v$ results in at most one negative edge (if $v=c$) and zero (if $v=a$ or $b$).

    If all of them are equal to $1$, and all of them are incident to two vertices of $R$. Write these edges as $xu,yv,zv$, where $u,v\in R$ are distinct and $x,y,z\in L$ are distinct. If $v=a$ or $b$, then switching along the degree cut of $v$ reduces the number of negative edges by two. If $v=c$, then switching along the degree cut of $v$ makes the edge $xv$ negative reducing the negative edges by $1$. Since two negative edges $xu,xv$ are incident to vertex $x\in L$, switching along the degree cut of $x$ reduces the negative edges by $1$ and we have a switching with at most two negative edges.

    If all of them are equal to $1$, and all of them are incident to three vertices forming a matching. Let $S$ be the cut obtained by swapping the principal cut along the matched edge of $a$. The cut $S$ contains only $6$ edges (as $4$ edges are removed but the edge $ab$ is added) out of which there are $4$ negative (matching and $ab$) edges, thus switching along the cut $S$ gives a signing with at most two negative edges.
\end{proof}

We next show that the both parts of the switching class, divided based on the positioning of the terminal $t$ with respect to $s$ have small frustration.

\begin{lem}\label{lem:maxis4}
    For every signing $\sigma$ of the gadget $H$, we have $\max(f_0(H,\sigma),f_1(H,\sigma))\leq 4$.
\end{lem}

\begin{proof}
By \Cref{lem:minis2}, we know that there exists a reduced signing $\gamma$ obtained from $\sigma$ with at most two negative edges after switching along $x$ with $x_s=0,x_t=\beta$. It remains to find a switching with $x_s=0$, $x_t=1-\beta$ and at most four negative edges. We may describe this as switching $t$ relative to $\gamma$, ensuring $s$ is not switched.

    If $\delta^-_\gamma(t)\geq 1$ then switching along the degree cut of $t$ increases the number of negative edges by at most $2$, proving that there exists a signing with at most $4$ negative edges in the other half of the switching class partitioned according to the terminal $t$.

    Thus $\delta^-_\gamma(t)=0$ and if the signing $\gamma$ contains at most one negative edge, then after switching along the degree cut of $t$, increases the number of negative edges by $3$ and we still have only $4$ negative edges.
    
    So, we have to prove the lemma when the signing $\gamma$ contains exactly two negative edges and $\delta^-_\gamma(t)=0$.

    \paragraph{Case $\gamma(ab)=0$:} This implies $\delta^-_\gamma(d)=\delta^-_\gamma(s)=1$ because if either one of them is two then the signing $\gamma$ can be reduced further by switching along the degree cut making the number of negative edges in the reduced signing $\gamma$ to be $1$.
    
    If both the negative edges of the signing $\gamma$ are incident to a single vertex $v\in R$, then switching along the degree cut of $t$ increases the number of negative edges by $3$. Note that there are $3$ negative edges incident to $v$, thus switching along the degree cut of $v$ decreases the number of negative edges by at least $2$ and we have at most $3$ negative edges left.
    
    If both the negative edges are incident to two different vertices $u,v\in R$, say $xu,yv$. If $c$ is one of $u,v$, then switching along the degree cut of $t$ not only increases the number of negative edges by $3$ but also induces two negative edges incident to $c$. So switching along the degree cut of $c$ reduces the number of negative edges by $1$. 

    Suppose now that the two negative edges are incident with two distinct vertices of $R$, and that these vertices are $a$ and $b$. Thus the two negative edges are $xa$ and $yb$, where $x,y\in \{s,d\}$. Since $\delta^-_\gamma(s)=\delta^-_\gamma(d)=1$, these two edges form a matching between $\{s,d\}$ and $\{a,b\}$.
    Let $w\in\{a,b\}$ be the vertex matched to $d$ and let $\bar{w}$ denote the other vertex in $\{a,b\}$. Thus the two negative edges of the signing $\gamma$ are $dw$ and $s\bar{w}$. After switching at $t$, the negative edges are $dw,s\bar{w},ta,tb,tc$. Now consider the cut $S=\delta(\{s,t,w\})$, swapped principal cut along $dw$. The cut has size $6$, with edges $\{s\bar{w},sc,t\bar{w},tc,dw,ab\}$. Among these six edges, the four edges $\{s\bar{w},t\bar{w},tc,dw\}$ are negative. Hence switching along the set $\{c,\bar{w},d\}$ decreases the number of negative edges by $2$ resulting in a total of $3$ negative edges. Note that $t$ is switched only once and $s$ is not switched during this process, thus falling in the other half of the switching class.

    \paragraph{Case $\gamma(ab)=1$:} We have that one of the negative edges is $ab$, $\delta^-_\gamma(t)=0$. Let the other negative edge be $xv$ is incident to $v\in R$. After switching along the degree cut of $t$ increases the number of negative edges by $3$. If $v=a$ or $b$ then there are three negative edges incident to $v$, ($xv,tv,ab$), else $v=c$ and there are two negative edges incident to $v$, ($xv,tv$). In either case, switching along the degree cut of $v$ reduces the number of negative edges by at least $1$ and we finally have at most $4$ negative edges.
\end{proof}

The frustration is at most $4$ in both halves of the switching class categorized with respect to the terminal $t$.

\begin{lem}\label{lem:goodsigning}
    There exists a signing $\sigma^*$ of the gadget $H$ that attains both $\min(f_0(H,\sigma^*),f_1(H,\sigma^*))= 2$ and $\max(f_0(H,\sigma^*),f_1(H,\sigma^*))=4$ simultaneously.
\end{lem}

\begin{proof}
    Consider the signing with $\sigma^*(ab)=\sigma^*(cd)=1$ and all other edges are zero. Since the signing itself has at most two negative edges with no switching, we have $f_0(H,\sigma^*)\leq 2$.
    
    Since $sabs$ is a $3$-cycle and $tbdct$ is a $4$-cycle with exactly one negative edge per cycle and any cut intersecting these cycle(s) intersects in at least two edges, switching along any intersecting cut retains at least one negative edge so $f_0(H,\sigma^*)\geq 2$ making it equal.

    Now consider $f_1(H,\sigma^*)$, that is we should consider all switching(s) that have $t$ on the opposite side of $s$. First, switching across the degree cut $t$ increases the number of negative edges by $3$ and there are exactly two negative edges incident to $c$, ($tc,dc$). Thus switching along the degree cut of $c$, reduces the number of negative edges to $4$, $f_1(H,\sigma^*)\leq 4$.
    
    For every vertex $v\in R$, any switching vector $x$ with $x_s=0$ and $x_t=1$, $x_v$ is either $0/1$. Thus at least one of the edge $tv$ or $sv$ is across the cut corresponding to $x$, making it negative after switching along $x$. So we have $f_1(H,\sigma^*)\geq 3$. We also have another disjoint circuit $abda$ with one negative edge. Any cut with non-empty intersection with the cycle intersects in exactly $2$ edges thus retaining at least one negative edge after switching. So $f_1(H,\sigma^*)\geq 4$ making it equal.
\end{proof}

\subsubsection{Reduction from Cubic Maximum Stable Set}\label{sec:stable-set-reduction}

Let $G=(V,E)$ be a cubic graph, let $n\coloneqq |V|$, let $m\coloneqq |E|$, and let $\alpha(G)$ be its independence number.

\paragraph{Reduction:}First we construct a gadget graph $H_{v_i}$ for every $v_i\in V$ with two distinguished terminals $s_{i}$ and $t_{i}$ labeling other vertices as $a_i,b_i,c_i,d_i$, see Figure~\ref{fig:31a}.

\begin{figure}[th!]
\centering
\begin{subfigure}[t]{0.11\textwidth}
    \centering
    \includegraphics[width=0.9\textwidth]{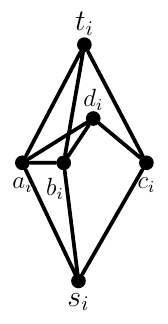}
    \caption{Instance of the gadget graph $H_{v_i}$ for $v_i\in V$.}
    \label{fig:31a}
\end{subfigure}\hfill
\begin{subfigure}[t]{0.28\textwidth}
    \centering
    \includegraphics[width=\textwidth]{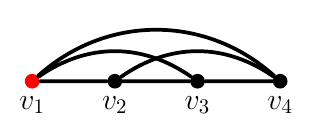}
    \caption{An instance of \textsc{Maximum Stable Set} on a connected cubic graph. The red vertex forms a maximum stable set.}
    \label{fig:31b}
\end{subfigure}\hfill
\begin{subfigure}[t]{0.55\textwidth}
    \centering
    \includegraphics[width=0.8\textwidth]{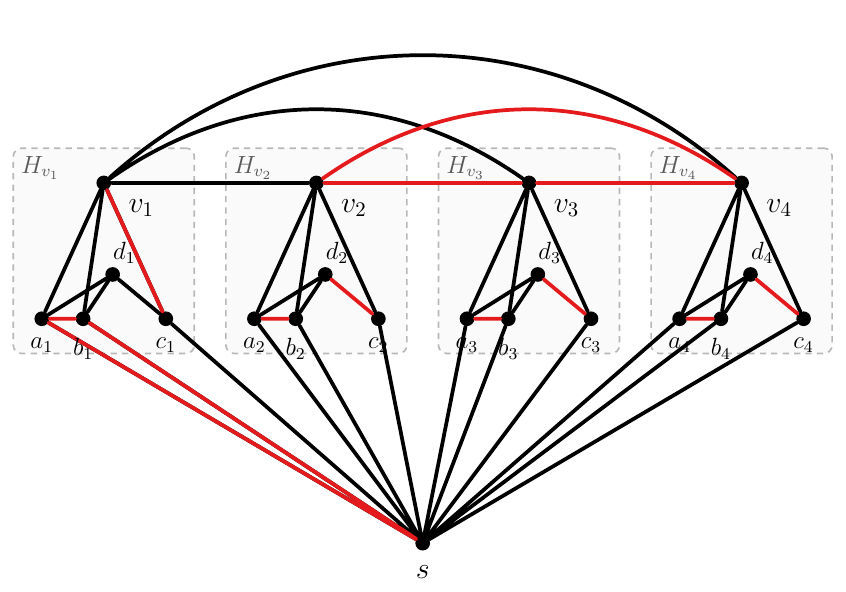}
    \caption{The corresponding instance of \textsc{Maximum Frustration Index} of a graph. The light dashed frames mark the non-shared part of the four gadget copies; all four copies share the terminal $s$.}
    \label{fig:31c}
\end{subfigure}\hfill
\caption{Illustration for the reduction and proof of \Cref{lem:upperBound}.}
\label{fig:31}
\end{figure}

Construct $G'$ from a copy of $G$ by adding a new root vertex $s$. For each $v_i\in V(G)$, attach a copy $H_{v_i}$ of the gadget $H$ by identifying its terminal $s_i$ with $s$ and its terminal $t_i$ with the copied $v_i$ of $G'$, see \Cref{fig:31b,fig:31c} for an example.

We have:
\[V'=V\uplus \{s\} \uplus \bigcup_{i=1}^n\{a_i,b_i,c_i,d_i\}\]
and \[E' = E\uplus \bigcup_{i=1}^n E(H_{v_i}).\]
Note that the construction is polynomial in $|V|+|E|$.

\paragraph{Correctness:}Now we are ready to describe the correctness of reduction by proving the following relation between the maximum frustration $F_{\max}(G')$ of the graph $G'$, and the independence number $\alpha(G)$ of graph $G$,
\[F_{\max}(G') = 2n+m-\alpha(G).\]

Recall that 
\[ F_{\max}(G') = \max_{\sigma\colon E(G')\to\{0,1\}} \min_{x\in\{0,1\}^{V(G')}} h(G',\sigma,x).\]

We prove upper and lower bounds separately.

\begin{lem}\label{lem:upperBound}
    $F_{\max}(G') \leq 2n+m-\alpha(G)$.\\
    Equivalently, for every signing $\sigma$ of the graph $G'$, there exists a switching $x$ such that \[h(G',\sigma,x_\sigma)\leq 2n+m-\alpha(G).\]
\end{lem}

\begin{proof}
Given a signing $\sigma$ of the graph $G'$, let $\sigma_i$ be the signing restricted to edges of the gadget $H_{v_i}$ part of $G'$ for $v_i\in V(G)$. For each gadget $H_{v_i}$ of $v_i\in V(G)$, let $\beta_i\in \{0,1\}$ such that $f_{\beta_i}(H_{v_i},\sigma_i)\leq 2$, guaranteed to exist according to \Cref{lem:minis2}.

Let the switching of $v_i$ (identified with the terminal $t_i$ of $H_{v_i}$) part of $G'$ be $\beta_i$. There exists a switching along the internal vertices $a_i,b_i,c_i,d_i$ such that the value $f_{\beta_i}(H_{v_i},\sigma_i)$ is attained. Choosing such a switching for the internal vertices, bounds the number of negative edges for the gadget $H_{v_i}$ to be at most $2$. Repeating the process for each gadget leads to at most $2n$ negative edges that are part of the gadget. Let $A\subset V(G)$ be the set of vertices $v_i$ with $\beta_i=1$.

Let $\sigma_G$ be the signing obtained from $\sigma$ by restricting to the edges of $G'$ copied from $G$. Let $\gamma$ be the signing of $G$, obtained from $\sigma_G$ after switching along the cut $A$ in $G$.

From \Cref{lem:switchingSet} construct $Y_{\gamma}\subset V$ and let $\Psi$ be the corresponding characteristic vector. Extend the switching vector $\Psi$ to the graph $G'$ by placing zeros for all other vertices of $G'$. Switching along the vector $\Psi$ changes the number of original (that are copies of $G$) negative edges to $|N_\gamma\Delta\delta_G(Y_\gamma)|$. 

Switching the terminals $t$ in the gadgets corresponding to the vertices in $Y_\gamma$ changes the number of negative edges in the gadget to at most $4$ after minimizing over all the switching(s) over the internal vertices of the corresponding gadget, according to \Cref{lem:maxis4}. So we have at most $2(n-|Y_\gamma|)+4|Y_\gamma|=2n+2|Y_\gamma|$ negative edges among the gadget edges.

According to \Cref{lem:switchingSet} we have $2|Y_{\gamma}|+|N_\gamma\Delta\delta_G(Y_\gamma)| \leq |E|-\alpha(G)$. So there are at most $2n+m-\alpha(G)$ negative edges after switching.
\end{proof}

\begin{lem}\label{lem:lowerBound}
    $F_{\max}(G') \geq 2n+m-\alpha(G)$.\\
    Equivalently, there exists a signing $\sigma^*$ of the graph $G'$, such that for every switching $x$, we have \[h(G',\sigma^*,x)\geq 2n+m-\alpha(G).\]
\end{lem}

\begin{proof}
    We construct such a signing $\sigma^*$ as follows. Assign all the original edges as negative. For each gadget assign negative sign to $ab,cd$ edges of the gadget and all others positive. Let this signing restricted to the gadget be $\sigma_i$ for all $i\in [|V(G)|]$.
    
    Fix a switching vector $x\in\{0,1\}^{|V(G')|}$ with $x_s=0$, and let $Y=\{v_i\in V(G')\colon x_{v_i}=1\}$. \Cref{lem:goodsigning} implies that the total number of negative gadget edges after switching along $x$ is at least $2n+2|Y|$. This is because, starting with the signing $ab,cd$ as negative edges, we have $f_1(H_{v_i},\sigma_i)=4$ and $f_0(H_{v_i},\sigma_i)=2$. This implies, for every gadget $H_{v_i}$, with $x_{v_i}=1$ we have at least $4$ negative edges and for every gadget $H_{v_i}$, with $x_{v_i}=0$ has at least $2$ negative edges for any switching on the internal vertices of the gadget. So, in total there are at least $2n+2|Y|$ negative gadget edges.
    
    The number of original negative edges after switching along $x$ is $|E\triangle \delta_G(Y)|$. Therefore total number of negative edges is at least $2n+2|Y|+|E\triangle \delta_G(Y)|$. Note that this bound holds varying over all switching vectors $x\in \{0,1\}^{|V(G')|}$ as the original edges and gadget edges interact only through the terminals $s,v_i \in V(G')$ which are conditioned based on $Y$.
    
    From \Cref{lem:stableSetIsGood} we know that varying $Y$ over all possible subsets, the count of negative edges, $2n+2|Y|+|E\triangle \delta_G(Y)|$ attains minimum at $Y=I$, the maximum stable set of $G$, and the minimum equals to $2n+m-\alpha(G)$. So $h(G',\sigma^*,x)\geq 2n+m-\alpha(G)$ for every switching $x$, in the switching class of $\sigma^*$ defined above.
\end{proof}

\begin{thm}\label{thm:frustration-hard}
It is NP-hard to compute the maximum frustration index of a graph.
\end{thm}

\begin{proof}
    Combining the upper and lower bounds, we obtain
\begin{equation*}
F_{\max}(G')=2n+m-\alpha(G).
\end{equation*}
Thus a polynomial-time algorithm for computing $F_{\max}(G')$ would yield a polynomial-time algorithm for computing $\alpha(G)$ on cubic graphs, which is NP-hard.
\end{proof}

The following corollary is immediate.

\begin{cor}
    Given a graph $G$, it is NP-hard to compute $\mu(M^*(G))$. 
\end{cor}

\subsubsection{Inapproximability}

The equivalence between maximum frustration and joins of cographic matroids transfers the hardness result to $\mu$. Combined with the gap hardness of Berman and Karpinski~\cite{berman2003improved} for stable sets in cubic graphs, the same reduction also gives the following inapproximability result.

\begin{thm}\label{thm:no-approx}
Given a cographic matroid $\cM$ by means of a graph representation $\cM=\cM^*(G)$, there is no polynomial-time algorithm that outputs a join $J$ with
\[
|J|\ge \frac{519}{520}\mu(\cM),
\]
unless $\mathrm{P}=\mathrm{NP}$.
\end{thm}

\begin{proof}
By the gap hardness of Berman and Karpinski~\cite{berman2003improved}, for every $\varepsilon>0$ it is NP-hard to distinguish, for a cubic graph $G$ on $176n$ vertices, between the cases
\[
\alpha(G)\ge (1-\varepsilon)98n
\qquad\text{and}\qquad
\alpha(G)\le (1+\varepsilon)97n .
\]

Fix $0<\varepsilon<1/101303$. Let $G$ be such a cubic graph, with $|V(G)|=176n$ and $|E(G)|=264n$. Construct $G'$ as in the proof of \Cref{thm:frustration-hard}, and let $\cM \coloneqq M^\ast(G')$. From the reduction,
\[
\mu(\cM)=F_{\max}(G')=2|V(G)|+|E(G)|-\alpha(G)=616n-\alpha(G).
\]
Thus, the case (yes-case) of $\alpha(G)\ge (1-\varepsilon)98n$, is equivalent to
\[
\mu(\cM)\le 616n-(1-\varepsilon)98n=(518+98\varepsilon)n,
\]
whereas in the case (no-case) of $\alpha(G)\le (1+\varepsilon)97n$, is equivalent to
\[
\mu(\cM)\ge 616n-(1+\varepsilon)97n=(519-97\varepsilon)n.
\]

Suppose there is a polynomial-time algorithm that outputs a join $J$ satisfying $|J|\ge \frac{519}{520}\mu(\cM)$. Since $J$ is a join, $|J|\le \mu(\cM)$. Therefore, in the yes-case, $|J|\le (518+98\varepsilon)n$, while in the no-case,
$|J|\ge \frac{519}{520}(519-97\varepsilon)n$.

By the choice of $\varepsilon$,
$\frac{519}{520}(519-97\varepsilon) > 518+98\varepsilon$. Hence the value of $|J|$ separates the two cases in polynomial time, contradicting the gap hardness. Therefore no such approximation algorithm exists unless $\mathrm{P}=\mathrm{NP}$.
\end{proof}

\subsection{Hardness of Sparse Paving Matroid Join}
\label{sec:sparsePavingHardness}

We give a separate reduction showing that the problem remains NP-hard for connected sparse paving matroids, even when the matroid is given explicitly by its list of bases.

\begin{thm}\label{thm:paving-mu-hard}
It is NP-hard to compute $\mu(\cM)$ for connected sparse paving matroids given by their list of bases.
\end{thm}

\begin{proof}
We reduce from the decision version of \textsc{Vertex Cover}: given a simple graph $G=(V,E)$ and an integer $K$, decide whether $\tau(G)\le K$. We may assume $0\le K\le |V|$.

For each edge $uv\in E$, introduce a new vertex $z_{uv}$. Let $\mathcal G$ be the $3$-uniform hypergraph with vertex set
$V(\mathcal G)=V\cup \{z_{uv}\colon uv\in E\}$ and hyperedge set
$\mathcal F(\mathcal G)=\bigl\{\{u,v,z_{uv}\}\colon uv\in E\bigr\}$.
Thus each edge $uv$ of $G$ is replaced by a hyperedge $\{u,v,z_{uv}\}$.

\begin{cl}\label{cl:Gprimecover}
$\tau(\mathcal G)=\tau(G)$.
\end{cl}

\begin{claimproof}
Let $C\subseteq V$ be a vertex cover of $G$. Then $C$ is also a vertex cover of $\mathcal G$, because every hyperedge of $\mathcal G$ has the form $\{u,v,z_{uv}\}$ for some $uv\in E$, and $C$ contains at least one of $u$ and $v$. Thus $\tau(\mathcal G)\le \tau(G)$.

Conversely, let $C$ be a vertex cover of $\mathcal G$. We modify $C$ so as to remove all vertices of the form $z_{uv}$. Suppose first that $z_{uv}\in C$ and at least one of $u$ and $v$ belongs to $C$. Then we may delete $z_{uv}$ from $C$, since $z_{uv}$ is contained only in the hyperedge $\{u,v,z_{uv}\}$, and this hyperedge is still covered by $u$ or $v$. Suppose next that $z_{uv}\in C$ and neither $u$ nor $v$ belongs to $C$. Then we may replace $z_{uv}$ by $u$; this preserves the vertex-cover property and does not increase the size.

Repeating this operation, we obtain a vertex cover $C'\subseteq V$ of $\mathcal G$ with $|C'|\le |C|$. Now let $uv\in E$. Since $C'$ covers the hyperedge $\{u,v,z_{uv}\}$ and $C'\subseteq V$, it must contain $u$ or $v$. Hence $C'$ is a vertex cover of $G$. Therefore $\tau(G)\le \tau(\mathcal G)$.
\end{claimproof}

We next construct a fixed linear $3$-uniform hypergraph with vertex cover number $7$.

\begin{cl}\label{cl:P13}
There exists a linear $3$-uniform hypergraph $\cP$ on $13$ vertices such that $\tau(P)=7$.
\end{cl}

\begin{claimproof}
Let $V(P)=\mathbb Z_{13}$, and let the hyperedges of $\cP$ be $\{i,i+1,i+4\}$ and $\{i,i+2,i+7\}$ for $i\in \mathbb Z_{13}$, where all subscripts are taken modulo $13$. Let $A=\{0,1,4\}$ and $B=\{0,2,7\}$. For a subset $X\subseteq \mathbb Z_{13}$, write $\Delta X=\{x-y\colon x,y\in X,\ x\ne y\}$. Then
\[
\Delta A=\{\pm1,\pm3,\pm4\}
\quad\text{and}\quad
\Delta B=\{\pm2,\pm5,\pm7\}=\{\pm2,\pm5,\pm6\}
\]
in $\mathbb Z_{13}$. Thus $\Delta A$ and $\Delta B$ partition the nonzero elements of $\mathbb Z_{13}$. Hence every unordered pair of vertices is contained in exactly one hyperedge of $\cP$, and $\cP$ is linear.

Let $S\subseteq V(\cP)$ be an independent set, and set $s=|S|$. Fix a vertex $t\in V(\cP)\setminus S$. The hyperedges containing $t$ and two vertices of $S$ are pairwise disjoint outside $t$, since $\cP$ is linear. Consequently, the corresponding pairs of vertices of $S$ are pairwise disjoint, and there are at most $\lfloor s/2\rfloor$ such pairs. Every pair of vertices of $S$ is contained in a unique hyperedge of $\cP$, and that third vertex is outside $S$ since $S$ is independent. Therefore $\binom{s}{2}\le (13-s)\big\lfloor s/2\big\rfloor$. If $s\ge 7$, then $\binom{s}{2}> (13-s)\big\lfloor s/2\big\rfloor$, a contradiction. Thus $s\le 6$, and hence $\alpha(\cP)\le 6$. The set $\{0,1,2,6,9,11\}$ contains no hyperedge of $\cP$, so $\alpha(P)\ge 6$. Hence $\alpha(\cP)=6$. By \Cref{thm:gallai}, $\tau(\cP)=13-\alpha(\cP)=7$.
\end{claimproof}

Let $N_0=|V(\mathcal G)|=|V|+|E|$, and set $\delta=2K-N_0+2$. Define
\[
a=1+\max\{0,-\delta\}
\quad\text{and}\quad
b=1+\max\{0,\delta\}.
\]
Then $a,b\ge 1$ and $b-a=\delta$. Let $\cP^1,\ldots,\cP^a$ be disjoint copies of $\cP$, and add $b$ further isolated vertices. Let $\mathcal H_0$ be the disjoint union of $\mathcal G$, the copies $\cP^1,\ldots,\cP^a$, and these $b$ isolated vertices. Set $n=|V(\mathcal H_0)|$ and $k=K+7a$.

\begin{cl}\label{cl:threshold}
The hypergraph $\mathcal H_0$ is linear and $3$-uniform, $n=2k+2$, and $\tau(\mathcal H_0)\le k$ if and only if $\tau(G)\le K$. Moreover, $a,b=O(N_0)$.
\end{cl}

\begin{claimproof}
The hypergraph $\mathcal G$ is linear because $G$ is simple and each new vertex $z_{uv}$ belongs to a unique hyperedge. Since $\cP$ is linear and the copies $\cP^1,\ldots,\cP^a$ are disjoint from each other and from $\mathcal G$, the hypergraph $\mathcal H_0$ is linear. It is also $3$-uniform, since all its hyperedges have size $3$.

By Claim~\ref{cl:Gprimecover} and Claim~\ref{cl:P13}, and since isolated vertices do not affect the vertex cover number,
\[
\tau(\mathcal H_0)=\tau(\mathcal G)+a\tau(\cP)=\tau(G)+7a.
\]
Also $n=N_0+13a+b$. Since $b-a=\delta$, we get
\[
n=N_0+14a+\delta=N_0+14a+(2K-N_0+2)=2K+14a+2=2k+2.
\]
Therefore $\tau(\mathcal H_0)\le k$ if and only if $\tau(G)+7a\le K+7a$, that is, if and only if $\tau(G)\le K$. Finally, since $0\le K\le |V|$, we have $|\delta|\le N_0+2$. Hence $a,b=O(N_0)$.
\end{claimproof}

We now construct the sparse paving matroid. Write $\mathcal H_0=(U,\mathcal F)$, where $U=V(\mathcal H_0)$ and $\mathcal F$ is the hyperedge set of $\mathcal H_0$. Let $r=n-3=2k-1$. For each $F\in \mathcal F$, set $H_F=U\setminus F$, and let
$\mathcal C=\{H_F\colon F\in \mathcal F\}$.
Thus each member of $\mathcal C$ has size $r$. We declare the sets in $\mathcal C$ to be the circuit-hyperplanes, and define
$\mathcal B=\textstyle\binom{U}{r}\setminus \mathcal C$.

\begin{cl}\label{cl:matroid}
The family $\mathcal B$ is the set of bases of a rank-$r$ sparse paving matroid $\cM$ on $U$. Moreover, the members of $\mathcal C$ are exactly the circuit-hyperplanes of $\cM$.
\end{cl}

\begin{claimproof}
Since $\mathcal H_0$ is linear, distinct hyperedges $F,F'\in\mathcal F$ satisfy $|F\cap F'|\le 1$. Hence
\[
|H_F\cap H_{F'}|=|U|-|F\cup F'|\le n-5=r-2.
\]
Also, since $b\ge 1$, the hypergraph $\mathcal H_0$ has an isolated vertex. This isolated vertex belongs to every set $H_F$, and therefore $\mathcal C\neq \binom{U}{r}$. By \Cref{thm:pav}, the family $\binom{U}{r}\setminus \mathcal C$ is the set of bases of a rank-$r$ sparse paving matroid on $U$, and the sets in $\mathcal C$ are precisely its circuit-hyperplanes.
\end{claimproof}

\begin{cl}\label{cl:connected}
The matroid $\cM$ is connected.
\end{cl}

\begin{claimproof}
Fix one copy $\cP_0$ of $\cP$ inside $\mathcal H_0$, and let $x,y\in U$ be distinct. We find a hyperedge $F$ of $\cP_0$ such that $x,y\notin F$. Each vertex of $\cP_0$ is contained in six hyperedges, while $\cP_0$ has $26$ hyperedges. Hence at most twelve hyperedges of $\cP_0$ meet $\{x,y\}$, and therefore some hyperedge $F$ of $\cP_0$ avoids both $x$ and $y$. Then $H_F=U\setminus F$ is a circuit-hyperplane of $\cM$ containing both $x$ and $y$. Therefore any two distinct elements of $\cM$ are contained in a common circuit, and $\cM$ is connected.
\end{claimproof}

The construction has polynomial size. Indeed, Claim~\ref{cl:threshold} gives $a,b=O(N_0)$. Moreover, $r=n-3$, so $\binom{U}{r}=\binom{U}{3}$ has size $O(n^3)$. Thus the family $\mathcal B=\binom{U}{r}\setminus \mathcal C$ can be constructed and listed in polynomial time. Hence the basis-list representation of $\cM$ has polynomial size.

\begin{cl}\label{cl:k-join-cover}
For every set $J\subseteq U$ of size $k$, $J$ is a join of $\cM$ if and only if $J$ is a vertex cover of $\mathcal H_0$.
\end{cl}

\begin{claimproof}
By Claim~\ref{cl:matroid}, the rank-$r$ circuits of $\cM$ are exactly the sets $H_F$ with $F\in\mathcal F$. Since $\cM$ is paving of rank $r=2k-1$, every circuit of $\cM$ has size $r$ or $r+1=2k$.

Suppose first that $J$ is a join of $\cM$ and $|J|=k$. For every $F\in \mathcal F$, the set $H_F$ is a circuit of size $r$, and hence
\[
|J\cap H_F|\le \big\lfloor r/2\big\rfloor=k-1.
\]
Since $|J|=k$, this implies $J\nsubseteq H_F$, or equivalently $J\cap F\neq \emptyset$. Thus $J$ meets every hyperedge of $\mathcal H_0$, so $J$ is a vertex cover of $\mathcal H_0$.

Conversely, suppose that $J$ is a vertex cover of $\mathcal H_0$ and $|J|=k$. Then $J\cap F\neq \emptyset$ for all $F\in\mathcal F$, and hence
\[
|J\cap H_F|\le |J|-1=k-1=\big\lfloor r/2\big\rfloor
\]
for every rank-$r$ circuit $H_F$. If $C$ is a circuit of size $r+1=2k$, then $|J\cap C|\le |J|=k=\lfloor |C|/2\rfloor$. Thus $J$ satisfies the join inequality for every circuit of $\cM$, and hence $J$ is a join.
\end{claimproof}

If $\tau(G)\le K$, then Claim~\ref{cl:threshold} gives $\tau(\mathcal H_0)\le k$. Hence $\mathcal H_0$ has a vertex cover of size exactly $k$, obtained by adding arbitrary vertices if necessary. By Claim~\ref{cl:k-join-cover}, $\cM$ has a join of size $k$, so $\mu(\cM)\ge k$.

We now show that no join has size larger than $k$. Let $J\subseteq U$ with $|J|=k+1$. Since $|U|=2k+2$ and $r=2k-1$, there is a set $X\subseteq U$ of size $r+1=2k$ with $J\subseteq X$. The set $X$ is dependent, so it contains a circuit $C$. Since $\cM$ is paving of rank $r$, the circuit $C$ has size $r$ or $r+1$. If $|C|=r+1$, then $C=X$ and $|J\cap C|=k+1>k=\lfloor |C|/2\rfloor$. If $|C|=r$, then $C$ omits at most one element of $J$, so $|J\cap C|\ge k>k-1=\lfloor |C|/2\rfloor$. In both cases $J$ violates the join condition. Hence no set of size $k+1$ is a join, and therefore no larger set is a join either. Thus $\mu(\cM)=k$.

Now suppose that $\tau(G)>K$. Then Claim~\ref{cl:threshold} gives $\tau(\mathcal H_0)>k$. Hence $\mathcal H_0$ has no vertex cover of size $k$, and by Claim~\ref{cl:k-join-cover}, $\cM$ has no join of size $k$. Therefore $\cM$ has no join of size greater than $k$, since every subset of a join is again a join. On the other hand, every set $J\subseteq U$ with $|J|=k-1$ is a join, because every circuit of $\cM$ has size $2k-1$ or $2k$. Hence $\mu(\cM)=k-1$.

We have thus proved that $\mu(\cM)=k$ if $\tau(G)\le K$, and $\mu(\cM)=k-1$ if $\tau(G)>K$. Thus a polynomial-time algorithm for computing $\mu(\cM)$ would decide \textsc{Vertex Cover}. This proves that computing $\mu(\cM)$ is NP-hard for connected sparse paving matroids given by their list of bases.
\end{proof}

\section{Comparison Bounds for Binary, Paving, and General Matroids}
\label{sec:other}

We have seen that Frank's equality does not extend to matroids in general; in fact, it can already fail for binary matroids. It is therefore natural to ask whether the two parameters can nevertheless be bounded in terms of each other. In this section we prove such comparisons for general matroids, with better guarantees for binary and paving matroids. In fact, we prove that for binary matroids, the failure is one-sided: the inequality $\mu(\cM)\leq \eta(\cM)$ still holds. 

\subsection{Binary Matroids}
\label{sec:binary}

In this section, we prove that $\mu(\cM)\leq \eta(\cM)$ for binary matroids. Let $\cM$ be a connected binary matroid not isomorphic to $U_{1,1}$, and let $C_1,\ldots,C_t$ be an ear decomposition of $\cM$ with lobes $P_1,\ldots,P_t$. Recall that $D_i=\bigcup_{j=1}^i P_j$ and $\cM_i=\cM|D_i$. We also set $D_0=\emptyset$, and let $\cM_0$ be the empty matroid.

\begin{lem}
\label{lem:upperBoundOnJoin}
For every $i\in[t]$, $\mu(\cM_i)\leq \mu(\cM_{i-1})+\lfloor |P_i|/2\rfloor$.
\end{lem}

\begin{proof}
Fix $i\in[t]$, and let $J$ be a maximum join of $\cM_i$. Then $J_{i-1}\coloneqq J\cap D_{i-1}$ is a join of $\cM_{i-1}$. If $|J\cap P_i|\leq \lfloor |P_i|/2\rfloor$, then $\mu(\cM_i)=|J|=|J\cap D_{i-1}|+|J\cap P_i|\leq \mu(\cM_{i-1})+\lfloor |P_i|/2\rfloor$, as required. Hence we may assume that $|J\cap P_i|> \lfloor |P_i|/2\rfloor$.

We use the equivalent formulation of joins in terms of conservative weightings. Let $w=w_J$ be the weighting of $D_i$ that assigns weight $-1$ to the elements of $J$ and weight $+1$ to the elements outside $J$. Since $J$ is a join, every circuit of $\cM_i$ has nonnegative $w$-weight.

Choose a circuit $C^\ast\in\cC(\cM_i)\setminus\cC(\cM_{i-1})$ of minimum $w$-weight. By the defining property of the lobe $P_i$, no circuit of $\cM_i$ contains a nonempty proper subset of $P_i$. Hence $P_i\subseteq C^\ast$. Set $A=C^\ast\setminus P_i$, so that $C^\ast=A\cup P_i$. Define $J'=J_{i-1}\triangle A$, and let $w'=w_{J'}$ be the corresponding weighting of $D_{i-1}$.

\begin{cl}
$J'$ is a join of $\cM_{i-1}$.
\end{cl}

\begin{claimproof}
Let $C$ be a circuit of $\cM_{i-1}$. Since $\cM_i$ is binary, the symmetric difference $C\triangle C^\ast$ is a disjoint union of circuits $C_1',\ldots,C_k'$ of $\cM_i$ by \Cref{thm:bin}. The set $C\triangle C^\ast$ contains $P_i$, and no circuit of $\cM_i$ contains a nonempty proper subset of $P_i$. Therefore exactly one of the circuits $C_1',\ldots,C_k'$ contains $P_i$; relabel so that this circuit is $C_1'$.

Since $C^\ast=A\cup P_i$ and $C\subseteq D_{i-1}$, we have $C\triangle C^\ast=(C\setminus A)\cup P_i\cup(A\setminus C)$. Therefore
\[
w(C\triangle C^\ast)-w(C^\ast)=w(C\setminus A)+w(P_i)+w(A\setminus C)-w(A)-w(P_i).
\]
Equivalently,
\[
\sum_{j=1}^k w(C_j')-w(C^\ast)=w(C\setminus A)-w(C\cap A).
\]
By the definition of $w'$, the right-hand side is $w'(C)$, and hence
\[
w'(C)=\sum_{j=1}^k w(C_j')-w(C^\ast).
\]
Now $C_1'$ is not contained in $D_{i-1}$, so the choice of $C^\ast$ gives $w(C_1')\geq w(C^\ast)$. Also, $w(C_j')\geq 0$ for every $j$, because $J$ is a join of $\cM_i$. Thus $w'(C)\geq 0$. Hence every circuit of $\cM_{i-1}$ has nonnegative $w'$-weight, and so $J'$ is a join of $\cM_{i-1}$.
\end{claimproof}

Since $|J\cap P_i|>\lfloor |P_i|/2\rfloor$, we have $w(P_i)<0$. Moreover, $|J\cap P_i|=(|P_i|-w(P_i))/2$. Since $0\leq w(C^\ast)=w(A)+w(P_i)$, we have $w(A)\geq -w(P_i)$. Switching the elements of $A$ changes the size of the join by $w(A)$, so $|J'|=|J_{i-1}|+w(A)$. Thus $\mu(\cM_{i-1})\geq |J'|\geq |J\cap D_{i-1}|-w(P_i)$, or equivalently $|J\cap D_{i-1}|\leq \mu(\cM_{i-1})+w(P_i)$.
Therefore
\begin{align*}
\mu(\cM_i)&=|J|\\
&=|J\cap D_{i-1}|+|J\cap P_i|\\
&\leq \mu(\cM_{i-1})+w(P_i)+\frac{|P_i|-w(P_i)}{2}\\
&=\mu(\cM_{i-1})+\frac{|P_i|+w(P_i)}{2}\\
&\leq \mu(\cM_{i-1})+\left\lfloor\frac{|P_i|}{2}\right\rfloor.
\end{align*}
The last inequality follows because $w(P_i)<0$ and $(|P_i|+w(P_i))/2$ is an integer.
\end{proof}

With the help of \Cref{lem:upperBoundOnJoin}, we now prove that the inequality $\mu(\cM)\leq \eta(\cM)$ holds.

\begin{thm}\label{thm:binary-bound}
Let $\cM$ be a connected binary matroid not isomorphic to $U_{1,1}$. Then $\mu(\cM)\leq \eta(\cM)$.
\end{thm}

\begin{proof}
Choose an ear decomposition $C_1,\ldots,C_t$ of $\cM$ with the maximum possible number $\gamma(\cM)$ of odd lobes, and let $P_1,\ldots,P_t$ be its lobes. We denote by $\cP_0$ and $\cP_1$ the sets of even and odd lobes, respectively. Applying \Cref{lem:upperBoundOnJoin} repeatedly gives
\begin{align*}
\mu(\cM)
&\leq \sum_{i=1}^t \left\lfloor\frac{|P_i|}{2}\right\rfloor\\
&= \sum_{P_i\in \cP_0} \frac{|P_i|}{2}
   + \sum_{P_i\in \cP_1} \frac{|P_i|-1}{2}\\
&= \frac{\sum_{i=1}^t |P_i|-\gamma(\cM)}{2}\\
&= \frac{|E(\cM)|-\gamma(\cM)}{2}\\
&= \eta(\cM).
\end{align*}
Hence $\mu(\cM)\leq \eta(\cM)$.
\end{proof}

\subsection{Paving Matroids}
\label{sec:paving}

Next, we give bounds on $\mu(\cM)$ and $\eta(\cM)$ for paving matroids. We begin with an explicit formula for $\phi(\cM)$, which also determines $\eta(\cM)$. Recall that a \emph{spanning circuit} is a circuit that spans the ground set; in a rank-$r$ matroid, such a circuit has size $r+1$.

\begin{lem}\label{lem:paving-eta-formula}
Let $\cM$ be a connected paving matroid not isomorphic to $U_{1,1}$, and let $r=r(\cM)$. Then
\begin{equation*}
\phi(\cM)=
\begin{cases}
1, & \text{if } r \text{ is odd},\\
0, & \text{if } r \text{ is even and } \cM \text{ has a spanning circuit},\\
2, & \text{if } r \text{ is even and } \cM \text{ has no spanning circuit}.
\end{cases}
\end{equation*}
Consequently,
\begin{equation*}
\eta(\cM)=
\begin{cases}
(r+1)/2, & \text{if } r \text{ is odd},\\
r/2, & \text{if } r \text{ is even and } \cM \text{ has a spanning circuit},\\
r/2+1, & \text{if } r \text{ is even and } \cM \text{ has no spanning circuit}.
\end{cases}
\end{equation*}
\end{lem}

\begin{proof}
Let $C_1,\ldots,C_t$ be an ear decomposition of $\cM$ with lobes $P_1,\ldots,P_t$. Since $\cM$ is paving, every circuit has size $r$ or $r+1$. Hence $|P_1|=|C_1|\in\{r,r+1\}$. By \Cref{lem:number-of-ears}, we have $\sum_{i=1}^t(|P_i|-1)=r$.

If $|P_1|=r+1$, then $|P_i|=1$ for every $i\geq 2$. If $|P_1|=r$, then exactly one of the lobes $P_2,\ldots,P_t$ has size $2$, and all the others have size $1$. Thus the possible lobe-size patterns are $r+1,1,\ldots,1$ and $r,2,1,\ldots,1$.

Suppose first that $r$ is odd. In both possible lobe-size patterns, there is exactly one even lobe. Hence every ear decomposition has exactly one even lobe, and therefore $\phi(\cM)=1$.

Now assume that $r$ is even. If $\cM$ has a spanning circuit, then this circuit has size $r+1$. By \Cref{thm:ear}, we may start an ear decomposition with this circuit. Then the lobe sizes are $r+1,1,\ldots,1$, so all lobes are odd. Hence $\phi(\cM)=0$.

Finally, suppose that $r$ is even and $\cM$ has no spanning circuit. Then no ear decomposition can start with an $(r+1)$-circuit, so every ear decomposition has lobe sizes $r,2,1,\ldots,1$. Thus every ear decomposition has exactly two even lobes, and therefore $\phi(\cM)=2$.

The stated values of $\eta(\cM)$ follow by substituting these values into $\eta(\cM)=(r(\cM)+\phi(\cM))/2$.
\end{proof}

We now use \Cref{lem:paving-eta-formula} to compare $\eta(\cM)$ with $\mu(\cM)$ for paving matroids.

\begin{thm}\label{thm:paving-bound}
Let $\cM$ be a connected paving matroid not isomorphic to $U_{1,1}$. Then
$$
\frac{2}{3}\mu(\cM)\le \eta(\cM)\le 2\mu(\cM).
$$
\end{thm}

\begin{proof}
Write $r=r(\cM)$. If $r=0$, then $\cM$ consists of a single loop, so $\eta(\cM)=\mu(\cM)=0$. If $r=1$, then $\cM$ is a parallel class, so $\eta(\cM)=1=\mu(\cM)$.

Suppose that $r=2$. Let $\{x,y\}$ be a basis of $\cM$. Since $\cM$ is connected, $\cM$ is not equal to $U_{2,2}$, and the elements of $\cM$ cannot be contained in only the two parallel classes of $x$ and $y$. Hence there exists an element $z$ that is parallel to neither $x$ nor $y$. Then $\{x,y,z\}$ is dependent, since $r=2$, but it contains no $2$-circuit. Thus $\{x,y,z\}$ is a $3$-circuit. Therefore every basis of $\cM$ is contained in a $3$-circuit. It follows that no $2$-element set is a join, while every singleton is a join. Hence $\mu(\cM)=1$. Moreover, $\cM$ has a spanning circuit, and \Cref{lem:paving-eta-formula} gives $\eta(\cM)=1$. Thus the claim holds for $r=2$. We may assume from now on that $r\ge 3$.

We first show $\eta(\cM)\le 2\mu(\cM)$. Let $B$ be a basis of $\cM$, and choose $J\subseteq B$ with $|J|=\lfloor r/2\rfloor$. Since $\cM$ is paving, every circuit of $\cM$ has size $r$ or $r+1$. Hence, for every circuit $C$,
$$
|J\cap C|\le |J|
=\Big\lfloor\frac r2\Big\rfloor
\le \Big\lfloor\frac{|C|}{2}\Big\rfloor.
$$
Thus $J$ is a join, and therefore $\mu(\cM)\ge \lfloor r/2\rfloor$. If $r$ is even, then \Cref{lem:paving-eta-formula} gives $\eta(\cM)\in\{r/2,r/2+1\}$. Since $r\ge 3$, we have $r\ge 4$, and hence
$$
\eta(\cM)\le \frac{r}{2}+1\le r
=2\Big\lfloor\frac r2\Big\rfloor
\le 2\mu(\cM).
$$
If $r$ is odd, then \Cref{lem:paving-eta-formula} gives
$$
\eta(\cM)=\frac{r+1}{2}
=\frac{r-1}{2}+1
\le r-1
=2\Big\lfloor\frac r2\Big\rfloor
\le 2\mu(\cM),
$$
where the inequality uses $r\ge 3$. This proves $\eta(\cM)\le 2\mu(\cM)$.

Now we prove $\eta(\cM)\ge \frac23\mu(\cM)$. Let $J$ be a maximum join of $\cM$. Since every join is independent, we may extend $J$ to a basis $B$ of $\cM$. As $\cM$ is connected and $r\ge 2$, it has no coloops, and hence $B\ne E(\cM)$. Choose an element $e\in E(\cM)\setminus B$, and let $C$ be the fundamental circuit of $e$ with respect to $B$. Since $\cM$ is paving and $C\subseteq B+e$, the circuit $C$ has size $r$ or $r+1$. If $|C|=r+1$, then $C=B+e$, and therefore
$$
\mu(\cM)=|J|=|J\cap C|\le \Big\lfloor \frac{r+1}{2}\Big\rfloor.
$$
If $|C|=r$, then $C=B-b+e$ for some $b\in B$. Since $J\subseteq B$, we have $|J\cap C|\ge \mu(\cM)-1$. Thus
$$
\mu(\cM)-1\le |J\cap C|\le \Big\lfloor\frac r2\Big\rfloor,
$$
and hence $\mu(\cM)\le \lfloor r/2\rfloor+1$. In either case, we have
$\mu(\cM)\le \lfloor r/2\rfloor+1$.

If $r$ is odd, then \Cref{lem:paving-eta-formula} gives $\eta(\cM)=(r+1)/2=\lfloor r/2\rfloor+1$, and hence $\eta(\cM)\ge \mu(\cM)$. If $r$ is even and $\cM$ has no spanning circuit, then $\eta(\cM)=r/2+1=\lfloor r/2\rfloor+1$, and again $\eta(\cM)\ge \mu(\cM)$. It remains to consider the case where $r$ is even and $\cM$ has a spanning circuit. Then $\eta(\cM)=r/2$, while $\mu(\cM)\le r/2+1$. Since $r\ge 4$ in this case, we get
$$
\frac{\eta(\cM)}{\mu(\cM)}
\ge
\frac{r/2}{r/2+1}
=
\frac{r}{r+2}
\ge
\frac{2}{3}.
$$
Thus $\eta(\cM)\ge \frac23\mu(\cM)$ in all cases.
\end{proof}

\begin{rem}
    The bounds stated in \Cref{thm:paving-bound} are tight.
    
    For the upper bound, the Fano matroid $F_7$ is a connected paving matroid of rank $3$, and any two of its elements are contained in a $3$-circuit. Hence no $2$-element set is a join, while every singleton is a join, so $\mu(F_7)=1$. By \Cref{lem:paving-eta-formula}, we have $\eta(F_7)=(3+1)/2=2$. Thus $\eta(F_7)=2\mu(F_7)$, showing that the constant $2$ cannot be improved. 
    
    For the lower bound, the rank-4 ternary matroid considered in \Cref{prop:incomparable} attains equality. Note that the matroid is a paving matroid: it has rank $4$ while all circuits have size either $4$ or $5$.
\end{rem}

\subsection{General Matroids}
\label{sec:general}

We now prove general bounds relating $\eta(\cM)$, $\mu(\cM)$, and the rank of $\cM$.

\begin{thm}\label{thm:general-bound}
Let $\cM$ be a connected matroid not isomorphic to $U_{1,1}$. Then
\begin{equation*}
\left\lceil \frac{\mu(\cM)}{2}\right\rceil
\le \eta(\cM)
\le
\left\lceil \frac{r(\cM)}{2}\right\rceil \mu(\cM).
\end{equation*}
\end{thm}

\begin{proof}
Write $r=r(\cM)$. If $r=0$, then $\cM$ consists of a single loop, so $\eta(\cM)=\mu(\cM)=0$. If $r=1$, then $\cM$ is a parallel class, so $\eta(\cM)=1=\mu(\cM)$. Hence we may assume that $r\ge 2$.

We first show $\left\lceil \frac{\mu(\cM)}{2} \right\rceil \le \eta(\cM)$. By the definition of $\gamma(\cM)$, there is an ear decomposition of $\cM$ with exactly $\gamma(\cM)$ odd lobes. Since every ear decomposition of $\cM$ has $|E(\cM)|-r$ lobes, we have $\gamma(\cM)\le |E(\cM)|-r$. Moreover, the parity of the number of odd lobes is the same as the parity of $|E(\cM)|$. Thus $|E(\cM)|-\gamma(\cM)$ is even and at least $r$. Hence $|E(\cM)|-\gamma(\cM)\ge 2\lceil r/2\rceil$. Using $\eta(\cM)=\bigl(|E(\cM)|-\gamma(\cM)\bigr)/2$, we obtain $\eta(\cM)\ge \bigl\lceil r/2\bigr\rceil$. Since every join is independent, $\mu(\cM)\le r$, and therefore $$\eta(\cM)\ge \left\lceil \frac r2\right\rceil \ge \left\lceil \frac{\mu(\cM)}{2}\right\rceil.$$

Next we prove $\eta(\cM)\le
\left\lceil \frac{r(\cM)}{2}\right\rceil \mu(\cM)$. First suppose that $\mu(\cM)\ge 2$. By the general inequality $\eta(\cM)\le r(\cM)$, we have $\eta(\cM)\le r$. Since $\mu(\cM)\ge 2$, we get 
$\left\lceil r/2\right\rceil \mu(\cM)
\ge
2\left\lceil r/2\right\rceil
\ge r$.
Thus $\eta(\cM)\le r\le \lceil r/2\rceil\mu(\cM)$. 

It remains to consider the case $\mu(\cM)=1$. Let $J$ be any independent $2$-set of $E(\cM)$. Since $\mu(\cM)=1$, the set $J$ is not a join. Hence there is a circuit $C$ of $\cM$ such that $|J\cap C|>|C|/2$. If $|J\cap C|\leq 1$, then this inequality forces $|C|=1$, contradicting the independence of $J$. Hence $|J\cap C|=2$. It follows that $|C|\leq 3$, while the independence of $J$ excludes $|C|=2$. Therefore $|C|=3$, and every independent $2$-set of $\cM$ is contained in a $3$-circuit.

Let $B$ be a basis of $\cM$. Since $|B|=r\geq 2$, \Cref{thm:triangle-circuit} implies that there is an element $e\in E(\cM)\setminus B$ such that $C_1=B+e$ is a circuit. Write $E(\cM)\setminus C_1=\{x_2,\ldots,x_s\}$ and set $D_1=C_1$. For each $j=2,\ldots,s$, choose a circuit $C_j\subseteq D_{j-1}+x_j$ with $x_j\in C_j$, and put $D_j=D_{j-1}\cup C_j$. Such a circuit exists because $B\subseteq D_{j-1}$ spans $\cM$. Then $C_j\setminus D_{j-1}=\{x_j\}$ for every $j\geq 2$, so $C_1,\ldots,C_s$ is an ear decomposition of $\cM$. Its lobe sizes are $r+1,1,\ldots,1$. Hence
\begin{equation*}
\phi(\cM)\leq
\begin{cases}
0, & \text{if } r \text{ is even},\\
1, & \text{if } r \text{ is odd}.
\end{cases}
\end{equation*}
Therefore
\begin{equation*}
\eta(\cM)=\frac{r+\phi(\cM)}{2}\leq \left\lceil\frac r2\right\rceil.
\end{equation*}
Since $\mu(\cM)=1$, this gives $\eta(\cM)\leq \lceil r/2\rceil\mu(\cM)$.
\end{proof}

\begin{rem} 
    The bounds stated in \Cref{thm:general-bound} are tight. 
    
    For the upper bound, let $\cB_r$ be the complete binary matroid of rank $r\ge 2$. Every pair of elements is contained in a $3$-circuit. Hence no $2$-element set is a join, while every singleton is a join, so $\mu(\cB_r)=1$. By the argument in the proof for the case $\mu=1$, we have $\eta(\cB_r)\le \lceil r/2\rceil$. On the other hand, the proof above gives $\eta(\cB_r)\ge \lceil r/2\rceil$. Thus $\eta(\cB_r)=\lceil r/2\rceil=\lceil r/2\rceil\mu(\cB_r)$, so the upper bound is attained. 
    
    For the lower bound, the rank-$4$ ternary matroid considered in \Cref{prop:incomparable} attains equality. For this matroid, $\mu=3$ and $\eta=2$, so $\eta=\lceil \mu/2\rceil$. 
\end{rem}

\section{Regular Matroids}
\label{sec:regular}

In this section, we prove a constant-factor upper bound on $\eta(\cM)$ in terms of $\mu(\cM)$ for regular matroids. The proof uses the graphic case, a cographic estimate, the exceptional matroid $R_{10}$, and Seymour's decomposition theorem.

\subsection{Preliminary Reductions}
\label{sec:preparations}

We start with a lemma on the effect of deleting parallel elements on $\mu$ and $\eta$.

\begin{lem}
\label{lem:parallel-deletion-mu-eta}
Let $\cM$ be a connected matroid of rank at least $2$, and let $e$ and $f$ be parallel elements of $\cM$. Set $\cM_0=\cM\backslash f$.
Then $\cM_0$ is connected. Furthermore, if $\cM_0$ is not isomorphic to $U_{1,1}$, then $\mu(\cM)=\mu(\cM_0)$ and $\eta(M) = \eta(\cM_0)$.
\end{lem}

\begin{proof}
Let $x,y\in E(\cM_0)$ be two distinct elements.
Since $\cM$ is connected, there is a circuit $C$ of $\cM$ containing both $x$ and $y$.
If $f\notin C$, then $C$ is a circuit of $\cM_0$.
If $f\in C$, then $e\notin C$, and $(C\setminus\{f\})\cup\{e\}$ is a circuit of $\cM_0$ containing $x$ and $y$.
Thus $\cM_0$ is a connected matroid. From now on, we assume that $\cM_0$ is not isomorphic to $U_{1,1}$.

We first show equality for $\mu$. Let $J_0$ be a join of $\cM_0$.
We claim that $J_0$ is also a join of $\cM$.
Let $C$ be a circuit of $\cM$.
If $f\notin C$, then $C$ is a circuit of $\cM_0$, so $|J_0\cap C|\le |C|/2$.
If $C=\{e,f\}$, then $|J_0\cap C|\le 1=|C|/2$.
If $f\in C$ and $e\notin C$, then $C-f+e$ is a circuit of $\cM_0$, so $|J_0\cap C|\le |J_0\cap (C-f+e)|\le |C|/2$.
Hence $J_0$ is a join of $\cM$, and then $\mu(\cM_0)\le \mu(\cM)$.

Conversely, let $J_1$ be a join of $\cM$.
If $f\notin J_1$, then $J_1$ is also a join of $\cM_0$.
If $f\in J_1$, then let $J_0=J_1-f+e$. 
Let $C$ be a circuit of $\cM_0$, and then $C$ is a circuit of $\cM$. 
If $e\notin C$, then $|J_0\cap C|\le |J_1 \cap C| \le \frac{|C|}{2}$. 
If $e\in C$, then $C-e+f$ is a circuit of $\cM$, so $|J_0\cap C| = |J_1 \cap (C-e+f)| \le \frac{|C|}{2}$.
Thus $J_0$ is a join of $\cM_0$ and $|J_0|=|J_1|$.
Hence $\mu(\cM)\le \mu(\cM_0)$, and therefore $\mu(\cM)=\mu(\cM_0)$.

Now we consider equality for $\eta$. We first show that $\eta(\cM)\le \eta(\cM_0)$.
Let $C_1,\dots,C_t$ be an ear decomposition of $\cM_0$ with $\gamma(\cM_0)$ odd lobes, and let $D_i=\bigcup_{k=1}^i C_k$ for $0\le i\le t$, where $D_0=\varnothing$.
Let $j$ be the least index such that $e\in D_j$.
Then $e\in C_j\setminus D_{j-1}$.
Insert the circuit $\{e,f\}$ immediately after $C_j$.
Its new lobe is the singleton $\{f\}$.
The ears before and including $C_j$ are unchanged.

For $i>j$, the only possible issue is  \ref{it:e3} relative to $D_{i-1}+f$.
Suppose that a circuit $C$ of $\cM$ satisfies \ref{it:e1} and \ref{it:e2} relative to $D_{i-1}+f$ and
$C\setminus(D_{i-1}+f)\subsetneq C_i\setminus D_{i-1}$.
If $f\notin C$, then $C$ is a circuit of $\cM_0$, contradicting \ref{it:e3}.
If $f\in C$, then $e\notin C$, and $C-f+e$ is a circuit of $\cM_0$.
Since $e\in D_{i-1}$, then $(C-f+e) \setminus D_{i-1}\subsetneq C_i\setminus D_{i-1}$ which contradicts \ref{it:e3}. 
Hence the new sequence is an ear decomposition of $\cM$ with one more odd lobe.
It follows that $\gamma(\cM)\ge \gamma(\cM_0)+1$, and therefore
$$\eta(\cM)=\frac{|E(\cM)|-\gamma(\cM)}2\le \frac{|E(\cM_0)|-\gamma(\cM_0)}2=\eta(\cM_0).$$

We next prove that $\eta(\cM_0)\le \eta(\cM)$.
Let $C_1,\dots,C_t$ be an ear decomposition of $\cM$ with $\gamma(\cM)$ odd lobes, and let $D_i=\bigcup_{k=1}^i C_k$ for $0\le i\le t$, where $D_0=\varnothing$. 
We modify this decomposition such that some circuit is the circuit $\{x,y\}$ with singleton lobe $\{x\}$, where $\{x,y\}=\{e,f\}$. 

If $C_1\ne\{e,f\}$, then let $x$ be the element of $\{e,f\}$ that appears later in the decomposition, let $y$ be the other one, and let $j$ be the least index such that $x\in D_j$.
Then $y\in D_{j-1}$.
The circuit $\{x,y\}$ satisfies \ref{it:e1} and \ref{it:e2} relative to $D_{j-1}$, then \ref{it:e3} gives 
$C_j\setminus D_{j-1}=\{x\}$. Thus we may replace $C_j$ by $\{x,y\}$, and the new sequence is still an ear decomposition of $\cM$.

If $C_1=\{e,f\}$, then $t\ge 2$ as the rank is at least 2.
Since $\{e,f\}$ is a $2$-circuit, $C_2$ contains exactly one of $e$ and $f$.
Let $y$ be that element, and let $x$ be the other one.
Then $C_2,\{x,y\},C_3,\dots,C_t$ is an ear decomposition of $\cM$.
The original first two lobe sizes are $2$ and $|C_2|-1$, while the new first two lobe sizes are $|C_2|$ and $1$.
If $|C_2|$ is odd, then this new ear decomposition would have two more odd lobes than the original one, contradicting the maximality of $\gamma(\cM)$.
Hence $|C_2|$ is even, and the new ear decomposition has the same number of odd lobes as the original one.

Therefore, we may assume that for some $j$ the ear $C_j$ is the circuit $\{x,y\}$ and its lobe is the singleton $\{x\}$, where $\{x,y\}=\{e,f\}$. 
Note that $\cM\backslash x\cong \cM_0$.
Since $x\notin D_{j-1}$, the circuits $C_1,\dots,C_{j-1}$ are circuits of $\cM\backslash x$.
Delete the circuit  $C_j$. 
For each $i>j$, define
$$
        \widetilde C_i=
        \begin{cases}
                C_i, & \text{if } x\notin C_i,\\
                C_i-x+y, & \text{if } x\in C_i.
        \end{cases}
$$
Since $x$ and $y$ are parallel elements, each $\widetilde C_i$ is a circuit of $\cM$ not containing $x$, and hence a circuit of $\cM\backslash x$.

Consider the circuit sequence $C_1,\dots,C_{j-1},\widetilde C_{j+1},\dots,\widetilde C_t$. 
For each $i>j$, the union of the previous circuits in this sequence is $D_{i-1}-x$.
If $x\notin C_i$, then
$\widetilde C_i\cap (D_{i-1}-x)=C_i\cap D_{i-1}\ne \varnothing$.
If $x\in C_i$, then $y\in \widetilde C_i\cap (D_{i-1}-x)$.
Also,
$\widetilde C_i\setminus (D_{i-1}-x)=C_i\setminus D_{i-1}$.
Hence \ref{it:e1} and \ref{it:e2} are preserved.
It remains to check \ref{it:e3}.
Suppose that a circuit $C$ of $\cM\backslash x$ satisfies \ref{it:e1} and \ref{it:e2} relative to $D_{i-1}-x$ and
$$
        C\setminus (D_{i-1}-x)
        \subsetneq
        \widetilde C_i\setminus (D_{i-1}-x)
        =
        C_i\setminus D_{i-1}.
$$
Since $x\in D_{i-1}$ and $x\notin C$, we have
$C\setminus D_{i-1}=C\setminus (D_{i-1}-x)$.
Thus $C$ is a circuit of $\cM$ satisfying \ref{it:e1} and \ref{it:e2} relative to $D_{i-1}$ and with
$C\setminus D_{i-1}\subsetneq C_i\setminus D_{i-1}$,
contradicting the minimality of $C_i$.
Therefore the new sequence is an ear decomposition of $\cM\backslash x$. 
Its lobes are the same as the original lobes, except that the odd singleton lobe $\{x\}$ is removed.
Hence it has $\gamma(\cM)-1$ odd lobes.
Therefore $\gamma(\cM\backslash x)\ge \gamma(\cM)-1$.
Since $\cM\backslash x\cong \cM_0$, we obtain $\gamma(\cM_0)\ge \gamma(\cM)-1$.
Thus
$$\eta(\cM_0)=\frac{|E(\cM_0)|-\gamma(\cM_0)}2\le \frac{(|E(\cM)|-1)-(\gamma(\cM)-1)}2=\eta(\cM).$$
Combining the two inequalities gives $\eta(\cM)=\eta(\cM_0)$, completing the proof.
\end{proof}

\subsection{Cographic Matroids}
\label{sec:cographic}

Now we prove one of the main results of this section, which gives a constant factor approximation of $\eta$ by $\mu$ for cographic matroids.

\begin{thm} \label{thm:cographic-eta-mu-bound} 
Let $h(x)=-x\log_2 x-(1-x)\log_2(1-x)$ denote the binary entropy function, and let $\alpha_0\in(0,1/2)$ be the unique solution of $h(\alpha_0)=1/3$. Set $c_0=1/(3\alpha_0)$, so that $c_0\approx 5.42$. Then every connected cographic matroid $\cM$ not isomorphic to $U_{1,1}$ satisfies $\eta(\cM)\le c_0\mu(\cM)$. 
\end{thm}

\begin{proof}
If $|E(\cM)|=1$, then, since $\cM\not\cong U_{1,1}$, the unique element is a loop. Hence $\mu(\cM)=\eta(\cM)=0$. If $r(\cM)=1$ and $|E(\cM)|\ge 2$, then $\cM$ is a parallel class, and so $\mu(\cM)=1$. Taking one $2$-circuit first and then adding the remaining elements by singleton lobes gives an ear decomposition with $|E(\cM)|-2$ odd lobes; this is best possible since the first lobe is even. Hence $\eta(\cM)=1$. Thus the theorem holds when $|E(\cM)|=1$ or $r(\cM)=1$.

Assume from now on that $|E(\cM)|\ge 2$ and $r(\cM)\ge 2$. By repeatedly applying \Cref{lem:parallel-deletion-mu-eta}, it suffices to prove the theorem when $\cM$ is simple. Write $\cM=M^*(G)$. Since $\cM$ is connected, all edges of $G$ lie in a single 2-edge-connected component; deleting isolated vertices if necessary, we may assume that $G$ is connected. Let $n=|V(G)|$, $m=|E(G)|$, and $r=r(\cM)=m-n+1$.

Since $\cM$ is simple, $G$ has no bond of size $1$ or $2$. Indeed, the circuits of $M^*(G)$ are precisely the bonds of $G$, so a bond of size $1$ would be a loop of $\cM$, and a bond of size $2$ would be a parallel pair of $\cM$. It follows that every vertex of $G$ has degree at least $3$: if some vertex had degree at most $2$, then the cut $\delta(v)$ would be a bond of size $2$. Hence $3n\le 2m$, and therefore
\begin{equation*}
        r=m-n+1\ge m-\frac{2m}{3}+1>\frac{m}{3}.
\end{equation*}
Set $t=r/m$. Since $G$ is connected and $n\ge 2$, we also have $r=m-n+1<m$. Thus $t\in(1/3,1)$.

Let $\cB=\cB(G)$ be the cutset code of $G$. Then we have $\mu(\cM)=\rho(\cB)$, where $\rho(\cB)$ is the covering radius of $\cB$. Since $G$ is connected, $\dim (\cB)=n-1=m-r$. The sphere covering bound~\cite{cohen1997covering} gives $2^{m-r}\sum_{i=0}^{\mu(\cM)}\binom{m}{i}\ge 2^m$,
and hence
\begin{equation}
\label{eq:sphere-covering-binomial-lower-bound}
        \sum_{i=0}^{\mu(\cM)}\binom{m}{i}\ge 2^r.
\end{equation}

For $s\in(0,1)$, let $\alpha(s)\in(0,1/2)$ be defined by $h(\alpha(s))=s$. This is well-defined because $h$ is continuous, $h(x)\to 0$ as $x\to 0^+$, $h(1/2)=1$, and $h$ is strictly increasing on $(0,1/2)$. In particular, $\alpha(t)$ is defined.

\begin{cl}
\label{claim:cographic-mu-alpha-lower-bound}
We have $\mu(\cM)\ge \alpha(t)m$.
\end{cl}

\begin{claimproof}
If $\mu(\cM)>m/2$, then the claim is immediate, since $\alpha(t)<1/2$. We may therefore assume that $\mu(\cM)\le m/2$. By \eqref{eq:sphere-covering-binomial-lower-bound}, we have $\mu(\cM)>0$. Let $x=\frac{\mu(\cM)}{m-\mu(\cM)}$.
Then $0<x\le 1$. Hence, for every integer $i$ with $0\le i\le \mu(\cM)$, we have $x^i\ge x^{\mu(\cM)}$. Therefore
\begin{equation*}
        (1+x)^m
        =
        \sum_{i=0}^m\binom{m}{i}x^i
        \ge
        \sum_{i=0}^{\mu(\cM)}\binom{m}{i}x^i
        \ge
        x^{\mu(\cM)}\sum_{i=0}^{\mu(\cM)}\binom{m}{i}.
\end{equation*}
It follows that
\begin{equation*}
        \sum_{i=0}^{\mu(\cM)}\binom{m}{i}
        \le
        x^{-\mu(\cM)}(1+x)^m.
\end{equation*}
Set $p=\mu(\cM)/m$. Since $x=p/(1-p)$ and $1+x=1/(1-p)$, we get
\begin{equation*}
        x^{-\mu(\cM)}(1+x)^m
        =
        \left(\frac{1-p}{p}\right)^{pm}
        \left(\frac{1}{1-p}\right)^m
        =
        p^{-pm}(1-p)^{-(1-p)m}
        =
        2^{m h(p)}.
\end{equation*}
Together with \eqref{eq:sphere-covering-binomial-lower-bound}, this gives
\begin{equation*}
        t=\frac{r}{m}\le h(p)=h\!\left(\frac{\mu(\cM)}{m}\right).
\end{equation*}
If $p=1/2$, then $p\ge \alpha(t)$. If $p<1/2$, then the monotonicity of $h$ on $(0,1/2)$ gives $p\ge \alpha(t)$. Thus $\mu(\cM)/m\ge \alpha(t)$, proving the claim.
\end{claimproof}

By Claim~\ref{claim:cographic-mu-alpha-lower-bound} and the general bound $\eta(\cM)\le r(\cM)$, we obtain
\begin{equation*}
        \eta(\cM)\le r = tm \le \frac{t}{\alpha(t)}\mu(\cM).
\end{equation*}
Define $g(\alpha)=h(\alpha)/\alpha$ for $0<\alpha<1/2$. Since $h'(\alpha)=\log_2(\frac{1-\alpha}{\alpha})$,
we have
\begin{equation*}
        g'(\alpha)
        =
        \frac{\alpha h'(\alpha)-h(\alpha)}{\alpha^2}
        =
        \frac{\log_2(1-\alpha)}{\alpha^2}
        <0.
\end{equation*}
Thus $g$ is strictly decreasing. Since $t>1/3$, we have $\alpha(t)>\alpha_0$, and hence
\begin{equation*}
        \frac{t}{\alpha(t)}
        =
        g(\alpha(t))
        <
        g(\alpha_0)
        =
        \frac{1}{3\alpha_0}
        =
        c_0.
\end{equation*}
Therefore $\eta(\cM)<c_0\mu(\cM)$, where $\alpha_0\approx 0.061$ and $c_0\approx 5.42$.
\end{proof}

\subsection{The Matroid \texorpdfstring{$R_{10}$}{R10}}
\label{sec:r10}

The matroid $R_{10}$ is one of the basic building blocks in Seymour's decomposition theorem for regular matroids. In what follows, we determine the values of the two parameters for $R_{10}$. We use the standard rank-$5$ binary representation
\begin{equation*}
R_{10}=M\!\left[
\begin{array}{ccccc|ccccc}
1&0&0&0&0&1&0&0&1&1\\
0&1&0&0&0&1&1&0&0&1\\
0&0&1&0&0&1&1&1&0&0\\
0&0&0&1&0&0&1&1&1&0\\
0&0&0&0&1&0&0&1&1&1
\end{array}\right].
\end{equation*}
Label the columns of the two blocks by
$e_1,\dots,e_5$ and $d_1,\dots,d_5$, respectively, with indices
taken modulo $5$.

\begin{prop}
\label{prop:r10-example}
The matroid $R_{10}$ satisfies $\mu(R_{10})=3$, $\gamma(R_{10})=4$, and
$\eta(R_{10})=3$. In particular, $\mu(R_{10})=\eta(R_{10})$.
\end{prop}

\begin{proof}
Using the above representation, the circuits of $R_{10}$ are the cyclic
shifts of the following six sets:
\begin{align*}
&\{e_i,e_{i+1},e_{i+2},d_i\},\\
&\{e_i,e_{i+3},d_i,d_{i+1}\},\\
&\{e_i,d_{i+1},d_{i+2},d_{i+4}\},\\
&\{e_i,e_{i+1},e_{i+2},e_{i+3},d_{i+2},d_{i+4}\},\\
&\{e_i,e_{i+1},e_{i+3},d_{i+1},d_{i+2},d_{i+3}\},\\
&\{e_i,e_{i+1},d_i,d_{i+1},d_{i+3},d_{i+4}\}.
\end{align*}
Thus there are fifteen circuits of size $4$ and fifteen circuits of
size $6$. Checking this list shows that $J=\{e_1,e_2,e_4\}$ meets
every $4$-circuit in at most two elements and every $6$-circuit in at
most three elements. Hence $J$ is a join, so $\mu(R_{10})\ge 3$. Conversely, the same list shows that every $4$-subset of $E(R_{10})$
is either a $4$-circuit or is contained in a $6$-circuit. Therefore no
$4$-set is a join, and no larger set is a join either. Hence
$\mu(R_{10})=3$.

Next consider the following sequence of circuits:
\begin{align*}
C_1&=\{e_1,e_2,e_3,e_4,d_3,d_5\},\\
C_2&=\{e_1,d_2,d_3,d_5\},\\
C_3&=\{e_2,e_4,d_4,d_5\},\\
C_4&=\{e_1,e_2,e_5,d_5\},\\
C_5&=\{e_1,e_2,e_3,d_1\}.
\end{align*}
These form an ear decomposition of $R_{10}$ with lobes
\begin{equation*}
\{e_1,e_2,e_3,e_4,d_3,d_5\},\ \{d_2\},\ \{d_4\},
\{e_5\},\ \{d_1\}.
\end{equation*}
Thus this ear decomposition has four odd lobes.

Since $|E(R_{10})|=10$ and $r(R_{10})=5$, every ear decomposition has
five ears by \Cref{lem:number-of-ears}. Moreover, every circuit of
$R_{10}$ has even size, so the first lobe in any ear decomposition is
even. Hence an ear decomposition has at most four odd lobes. Therefore
$\gamma(R_{10})=4$ and $\eta(R_{10})=(10-4)/2=3$.
\end{proof}

\subsection{The Regular-Matroid Bound}
\label{sec:regular_bounds}

Since every regular matroid is binary, \Cref{thm:binary-bound} shows that $\mu(\cM)\leq\eta(\cM)$ whenever $\cM$ is connected and not isomorphic to $U_{1,1}$. In this section, we establish the reverse inequality up to an absolute constant. We first give several auxiliary results that will be used in the proof of \Cref{thm:regular}. Recall that a cycle of a matroid is a possibly empty disjoint union of circuits.

\begin{lem}\label{lem:join-cycle}
Let $\cM$ be a matroid.
If $J$ is a join of $\cM$, then $|J\cap X|\le |X|/2$ for every cycle $X$ of $\cM$.
\end{lem}

\begin{proof}
Since $X$ is a cycle, it can be written as a disjoint union
$X=C_1\cup\dots\cup C_s$ of circuits. Because $J$ is a join,
$|J\cap C_j|\le |C_j|/2$ for $1\le j\le s$. Summing these inequalities gives
$$|J\cap X|=\sum_{j=1}^s |J\cap C_j|\le \frac12\sum_{j=1}^s |C_j|=|X|/2.$$
\end{proof}

\begin{lem}\label{lem:cycles-descend-under-contraction}
Let $\cM$ be a binary matroid, and let $S\subseteq E(\cM)$.
If $X$ is a cycle of $\cM$, then $X\setminus S$ is a cycle of $\cM/S$.
\end{lem}

\begin{proof} 
Let $r=r(\cM)$, and choose a representation of $\cM$ over $\mathbb F_2$ with columns $v_e\in\mathbb F_2^r$ for $e\in E(\cM)$. By \Cref{thm:bin}, a set $X\subseteq E(\cM)$ is a cycle of $\cM$ if and only if $\sum_{e\in X} v_e=0$. 
Let $W=\spa_{\mathbb F_2}\{v_s\colon s\in S\}$, and let $\pi\colon\mathbb F_2^r\to \mathbb F_2^r/W$ be the quotient map. The contraction $\cM/S$ is represented by the vectors $\bar v_e=\pi(v_e)$ for $e\in E(\cM)\setminus S$. Set $Y=X\setminus S$. Since $X$ is a cycle of $\cM$, we have $\sum_{e\in X}v_e=0$. Applying $\pi$, we get $\sum_{e\in X}\pi(v_e)=0$. Moreover, $\pi(v_s)=0$ for every $s\in S$. Hence $\sum_{e\in Y}\bar v_e=0$. Therefore $Y$ is a cycle of $\cM/S$. 
\end{proof}

\begin{lem}\label{lem:two-sum-mu-rank}
Let $\cM=\cM_1\oplus_2 \cM_2$ be a $2$-sum of binary matroids, and let
$E(\cM_1)\cap E(\cM_2)=\{p\}$. For $i\in\{1,2\}$, set $Q_i=M_i/p$. 
Then $\mu(\cM)\ge \mu(Q_1)+\mu(Q_2)$ and $r(\cM)=r(Q_1)+r(Q_2)+1$.
\end{lem}

\begin{proof}
Let $J_i$ be a maximum join of $Q_i$, so that $|J_i|=\mu(Q_i)$ for $i\in\{1,2\}$. Define $J\coloneqq J_1\cup J_2\subseteq E(\cM)$. We show that $J$ is a join of $\cM$. This will imply $\mu(\cM)\ge |J|=|J_1|+|J_2|=\mu(Q_1)+\mu(Q_2)$.

Let $C$ be a circuit of $\cM$. Then $C$ is of one of the following three types.

\medskip

\noindent
\textbf{Case 1.} $C$ is a circuit of $\cM_1\backslash p$.

Then $C$ is a cycle of $\cM_1$ disjoint from $\{p\}$. By \Cref{lem:cycles-descend-under-contraction}, $C$ is a cycle of $Q_1=\cM_1/p$. Since $J_1$ is a join of $Q_1$, \Cref{lem:join-cycle} gives $|J\cap C|=|J_1\cap C|\le |C|/2$.

\medskip

\noindent
\textbf{Case 2.} $C$ is a circuit of $\cM_2\backslash p$.

The same argument gives $|J\cap C|=|J_2\cap C|\le |C|/2$.

\medskip

\noindent
\textbf{Case 3.} $C=(C_1\cup C_2)-p$, where $C_i$ is a circuit of $\cM_i$ containing $p$, for $i\in\{1,2\}$.

By \Cref{lem:cycles-descend-under-contraction}, the set $C_i-p$ is a cycle of $Q_i=\cM_i/p$. Therefore, by \Cref{lem:join-cycle}, $|J_i\cap (C_i-p)|\le (|C_i|-1)/2$ for $i\in\{1,2\}$. Hence
\[
        |J\cap C|
        =
        |J_1\cap (C_1-p)|+|J_2\cap (C_2-p)|
        \le
        \frac{|C_1|-1}{2}+\frac{|C_2|-1}{2}
        =
        \frac{|C|}{2}.
\]
Thus the join inequality holds for every circuit $C$ of $\cM$. Therefore $J$ is a join of $\cM$, and $\mu(\cM)\ge \mu(Q_1)+\mu(Q_2)$.\\

For the rank formula, we have $r(\cM)=r(\cM_1)+r(\cM_2)-1$ and $r(Q_i)=r(\cM_i/p)=r(\cM_i)-1$. Therefore $r(\cM)=r(Q_1)+r(Q_2)+1$.
\end{proof}

\begin{lem}
\label{lem:three-sum-mu-rank}
Let $\cM=\cM_1\oplus_3 \cM_2$ be a $3$-sum of binary matroids, and let
$E(\cM_1)\cap E(\cM_2)=T$, where $T$ is a triangle of both $\cM_1$ and $\cM_2$.
For $i\in\{1,2\}$, set $Q_i=\cM_i/T$. Then
$\mu(\cM)\ge \mu(Q_1)+\mu(Q_2)$ and
$r(\cM)=r(Q_1)+r(Q_2)+2$.
\end{lem}

\begin{proof}
Let $J_i$ be a maximum join of $Q_i$, so that $|J_i|=\mu(Q_i)$ for $i\in\{1,2\}$. Define $J\coloneqq J_1\cup J_2\subseteq E(\cM)$. We show that $J$ is a join of $\cM$. This will imply $\mu(\cM)\ge |J|=|J_1|+|J_2|=\mu(Q_1)+\mu(Q_2)$. 

We verify the join inequality for all cycles of $\cM$; this is enough, since every circuit is a cycle. By the definition of the binary $3$-sum, every cycle $C$ of $\cM$ is of one of the following types. Let $C$ be a cycle of $\cM$. We distinguish three cases.

\medskip

\noindent
\textbf{Case 1.} $C$ is a cycle of $\cM_1$ disjoint from $T$.

By \Cref{lem:cycles-descend-under-contraction}, $C$ is a cycle of $Q_1=\cM_1/T$. Since $J_1$ is a join of $Q_1$, \Cref{lem:join-cycle} gives $|J\cap C|=|J_1\cap C|\le |C|/2$.

\medskip

\noindent
\textbf{Case 2.} $C$ is a cycle of $\cM_2$ disjoint from $T$.

The same argument gives $|J\cap C|=|J_2\cap C|\le |C|/2$.

\medskip

\noindent
\textbf{Case 3.} $C=C_1\triangle C_2$, where $C_i$ is a cycle of $\cM_i$, for $i=1,2$, and $C_1\cap T=C_2\cap T$.

In this case, $C=(C_1\setminus T)\cup(C_2\setminus T)$. By \Cref{lem:cycles-descend-under-contraction}, the set $C_i\setminus T$ is a cycle of $Q_i=\cM_i/T$. Since $J_i$ is a join of $Q_i$, \Cref{lem:join-cycle} gives $|J_i\cap (C_i\setminus T)|\le |C_i\setminus T|/2$ for $i\in\{1,2\}$. Hence
\[
        |J\cap C|
        =
        |J_1\cap (C_1\setminus T)|+|J_2\cap (C_2\setminus T)|
        \le
        \frac{|C_1\setminus T|}{2}+\frac{|C_2\setminus T|}{2}
        =
        \frac{|C|}{2}.
\]
Thus the join inequality holds for every circuit of $\cM$, and so $J$ is a join. Therefore $\mu(\cM)\ge \mu(Q_1)+\mu(Q_2)$.\\

For the rank formula, we use the standard equality $r(\cM)=r(\cM_1)+r(\cM_2)-2$ for a $3$-sum. Since $T$ is a triangle in each $\cM_i$, we have $r_{\cM_i}(T)=2$, and hence $r(Q_i)=r(\cM_i/T)=r(\cM_i)-2$. Therefore $r(\cM)=r(Q_1)+r(Q_2)+2$.
\end{proof}

Combining the preceding statements with Seymour's decomposition theorem yields the following bound.

\begin{thm}\label{thm:regular}
Let $\cM$ be a regular matroid, and let $\kappa(\cM)$ denote the number of connected components of $\cM$ with positive rank. Then $r(\cM)\le 6\mu(\cM)-2\kappa(\cM)$.
\end{thm}

\begin{proof}
We prove the theorem by induction on $|E(\cM)|$. 
If $E(\cM)=\{e\}$, then either $e$ is a loop, in which case $r(\cM)=\mu(\cM)=\kappa(\cM)=0$, or $e$ is not a loop, in which case $r(\cM)=1$, $\mu(\cM)=1$, and $\kappa(\cM)=1$. Thus the result holds when $|E(\cM)|\le 1$. Assume from now on that $|E(\cM)|\ge 2$, and that the theorem holds for every regular matroid with fewer elements. 

First suppose that $\cM$ is disconnected, say $\cM=\cM_1\oplus_1\cdots\oplus_1\cM_t$, where $\cM_1,\ldots,\cM_t$ are the connected components. Since the circuits of a direct sum are precisely the circuits of its components, we have $\mu(\cM)=\sum_{j=1}^t\mu(\cM_j)$. Also $r(\cM)=\sum_{j=1}^t r(\cM_j)$ and $\kappa(\cM)=\sum_{j=1}^t\kappa(\cM_j)$. Applying the induction hypothesis to each component and summing gives $r(\cM)\le 6\mu(\cM)-2\kappa(\cM)$.

We may therefore assume that $\cM$ is connected. Since $|E(\cM)|\ge 2$, the case $r(\cM)=0$ cannot occur. If $r(\cM)=1$, then $\cM$ is a parallel class, so $\mu(\cM)=1$ and $\kappa(\cM)=1$, and hence $r(\cM)=1\le 4=6\mu(\cM)-2\kappa(\cM)$. Thus we may assume that $r(\cM)\ge 2$. If $\cM$ has a parallel pair $e,f$, then \Cref{lem:parallel-deletion-mu-eta} implies that $\cM\setminus f$ is connected and satisfies $r(\cM\setminus f)=r(\cM)$ and $\mu(\cM\setminus f)=\mu(\cM)$. Since both matroids have positive rank and are connected, $\kappa(\cM)=\kappa(\cM\setminus f)=1$. The induction hypothesis applied to $\cM\setminus f$ gives $r(\cM)\le 6\mu(\cM)-2$. Hence we may assume that $\cM$ is simple. From now on, $\kappa(\cM)=1$, so it is enough to prove $r(\cM)\le 6\mu(\cM)-2$. Also, since $r(\cM)\ge 2$, the matroid $\cM$ is not isomorphic to $U_{1,1}$.

We now consider the basic classes. If $\cM$ is graphic, then \Cref{thm:frank} applies and gives $\mu(\cM)=\eta(\cM)$. Since $\eta(\cM)=(r(\cM)+\varphi(\cM))/2\ge r(\cM)/2$, we get $r(\cM)\le 2\mu(\cM)\le 6\mu(\cM)-2$. If $\cM$ is cographic, then the proof of \Cref{thm:cographic-eta-mu-bound} gives $r(\cM)<5.5\mu(\cM)$ for simple connected cographic matroids of rank at least $2$. If $\mu(\cM)\ge 2$, then $r(\cM)<5.5\mu(\cM)\le 6\mu(\cM)-1$, and since $r(\cM)$ is an integer, this implies $r(\cM)\le 6\mu(\cM)-2$. It remains to consider the case $\mu(\cM)=1$. Write $\cM=\cM^*(G)$, where $G$ is connected, and let $n=|V(G)|$ and $m=|E(G)|$. Since $\cM$ is simple, every bond of $G$ has size at least $3$, and hence $3n\le 2m$. Also $r(\cM)=m-n+1$, so $m\le 3r(\cM)-3$. The sphere-covering bound with covering radius $1$ gives $2^{r(\cM)}\le 1+m\le 3r(\cM)-2$. Hence $r(\cM)\le 2$, and therefore $r(\cM)\le 2\le 6\mu(\cM)-2$. If $\cM\cong R_{10}$, then $r(\cM)=5$ and $\mu(\cM)=3$, so $r(\cM)=5\le 6\mu(\cM)-2$.

We may now assume that $\cM$ is simple, connected, regular, and neither graphic nor cographic nor isomorphic to $R_{10}$. Suppose first that $\cM$ is not $3$-connected. By \Cref{thm:seym-2}, we may write $\cM=\cM_1\oplus_2 \cM_2$, where $\cM_1$ and $\cM_2$ are regular matroids and each $\cM_i$ has at least three elements. Let $E(\cM_1)\cap E(\cM_2)=\{p\}$ and set $Q_i=\cM_i/p$ for $i\in\{1,2\}$. Then each $Q_i$ is regular and has fewer elements than $\cM$. By \Cref{lem:two-sum-mu-rank}, $r(\cM)=r(Q_1)+r(Q_2)+1$ and $\mu(\cM)\ge \mu(Q_1)+\mu(Q_2)$. We claim that each $Q_i$ has positive rank. Indeed, since $\cM$ is connected, each $\cM_i$ is connected. If $r(Q_i)=0$, then $r(\cM_i)=1$. A connected rank-$1$ matroid is a parallel class; since $|E(\cM_i)|\ge 3$, the matroid $\cM_i\setminus p$ contains a $2$-circuit. This $2$-circuit is also a circuit of $\cM$, contradicting the simplicity of $\cM$. Thus $r(Q_i)>0$, so $\kappa(Q_i)\ge 1$. By the induction hypothesis, $r(Q_i)\le 6\mu(Q_i)-2\kappa(Q_i)$ for $i\in\{1,2\}$. Combining this with $\kappa(Q_i)\ge 1$ gives $r(\cM)\le 6(\mu(Q_1)+\mu(Q_2))-3\le 6\mu(\cM)-3\le 6\mu(\cM)-2$.

Finally suppose that $\cM$ is $3$-connected. Since $\cM$ is regular and is neither graphic nor cographic nor isomorphic to $R_{10}$, \Cref{thm:seym-3} gives a $3$-sum $\cM=\cM_1\oplus_3 \cM_2$ of regular matroids over a common triangle $T=E(\cM_1)\cap E(\cM_2)$ such that $|E(\cM_i)\setminus \clo_{\cM_i}(T)|\ge 6$ for $i\in\{1,2\}$. Set $Q_i=\cM_i/T$. Then each $Q_i$ is a regular minor with fewer elements than $\cM$. Choose $e_i\in E(\cM_i)\setminus\clo_{\cM_i}(T)$. Then $e_i$ is not a loop of $Q_i$, so $r(Q_i)>0$ and hence $\kappa(Q_i)\ge 1$. By the induction hypothesis, $r(Q_i)\le 6\mu(Q_i)-2\kappa(Q_i)$ for $i\in\{1,2\}$. Combining this with \Cref{lem:three-sum-mu-rank}, we obtain $r(\cM)=r(Q_1)+r(Q_2)+2\le 6(\mu(Q_1)+\mu(Q_2))-2\le 6\mu(\cM)-2$. This completes the proof of the theorem.
\end{proof}

\Cref{thm:regular} also yields the corresponding estimate for $\eta$, using the general inequality $\eta(\cM)\le r(\cM)$, shown in \eqref{eq:eta-rank-phi}.

\begin{cor}\label{cor:regular}
Let $\cM$ be a connected regular matroid not isomorphic to $U_{1,1}$. Then either $\cM$ is the one-element loop, in which case $\eta(\cM)=\mu(\cM)=0$, or $\eta(\cM)\le 6\mu(\cM)-2$.
\end{cor}

\begin{proof}
If $\cM$ is the one-element loop, then its unique ear decomposition has one odd lobe, so $\eta(\cM)=0$. Also $\mu(\cM)=0$. We may therefore assume that $\cM$ is not the one-element loop. Since $\cM$ is connected and not isomorphic to $U_{1,1}$, the parameter $\eta(\cM)$ is defined. Moreover, $\cM$ has positive rank, and hence $\kappa(\cM)=1$. By \Cref{thm:regular}, $r(\cM)\le 6\mu(\cM)-2$. Using \eqref{eq:eta-rank-phi}, we get $\eta(\cM)\le 6\mu(\cM)-2$.
\end{proof}

\section{Conclusion}
\label{sec:conclusion}

In this paper, we studied how far Frank's join--ear min--max theorem for graphic matroids extends beyond the graphic case. We showed that the equality $\mu(\cM)=\eta(\cM)$ does not extend to matroids in general: it fails already for cographic matroids, the two parameters are incomparable for general matroids, and the class of matroids satisfying the equality is not minor-closed. We also proved that computing $\mu(\cM)$ is NP-hard for cographic matroids, hard to approximate within a factor of $519/520$, and NP-hard for connected sparse paving matroids given by their list of bases. On the positive side, we established comparison bounds between $\mu(\cM)$ and $\eta(\cM)$ for several natural classes, including binary, paving, cographic, regular, and arbitrary connected matroids. We close the paper by mentioning some open problems.

\begin{enumerate}\itemsep0em
    \item Frank's theorem gives an exact equality for graphic matroids, while our examples show that this equality does not extend to cographic or paving matroids in general. It would be interesting to identify natural non-graphic classes for which the equality still holds. More broadly, can one characterize a reasonably large class of matroids satisfying $\mu(\cM)=\eta(\cM)$?

    \item For regular matroids, we proved the bound $\eta(\cM)\le 6\mu(\cM)-2$. This is unlikely to be best possible. What is the smallest constant $c_{\mathrm{reg}}$ such that $\eta(\cM)\le c_{\mathrm{reg}}\mu(\cM)$ holds for every connected regular matroid? In particular, does one always have $\eta(\cM)\le 2\mu(\cM)$?

    \item Another natural class to consider is the class of transversal matroids. Do connected transversal matroids satisfy $\mu(\cM)=\eta(\cM)$? If not, are the two parameters comparable by an absolute constant on this class?

    \item For matroids represented over finite fields of characteristic $2$, the ear-decomposition parameter $\varphi(\cM)$ can be computed in randomized polynomial time, while our results show that the join parameter $\mu(\cM)$ is hard to compute already for cographic and sparse paving matroids. This leaves open the algorithmic picture for other natural matroid classes. For which matroid classes can $\mu(\cM)$ be computed or approximated efficiently?

    \item Ear decompositions have many structural and algorithmic applications in graph theory beyond Frank's theorem. Some of these have matroidal analogues, such as the factor-critical theory of Szegedy and Szegedy for matroids representable over fields of characteristic $2$. It would be interesting to understand which further results about optimized or parity-constrained ear decompositions have meaningful extensions to matroids.
\end{enumerate}

\medskip
\paragraph{Acknowledgement.}

Yuhang Bai was supported by the National Natural Science Foundation of China (12131013 and 12471334), by China Scholarship Council (202406290002), and by Shaanxi Fundamental Science Research Project for Mathematics and Physics (22JSZ009). The research received further support from the Lend\"ulet Programme of the Hungarian Academy of Sciences (LP2021-1/2021), from the Ministry of Innovation and Technology of Hungary from the National Research, Development and Innovation Fund (ADVANCED 150556, ADVANCED 153096, and ELTE TKP 2021-NKTA-62), and from the Dynasnet European Research Council Synergy project (ERC-2018-SYG 810115).

\bibliographystyle{abbrv}
\bibliography{join}

\end{document}